\documentclass[10pt]{amsart}
\usepackage{fancyhdr}
\usepackage{stmaryrd}
\usepackage{calrsfs}

\usepackage{amsmath}
\usepackage[nospace,noadjust]{cite}
\usepackage{amsfonts}
\usepackage{amssymb,enumerate}
\usepackage{amsthm}
\usepackage{cite}
\usepackage{comment}
\usepackage{color}
\usepackage[all]{xy}
\usepackage{hyperref}
\usepackage{lineno}
\usepackage{tikz-cd}

\newtheorem{theorem}{Theorem}[section]
\newtheorem{prop}[theorem]{Proposition}
\newtheorem{question}[theorem]{Question}

\newtheorem{lemma}[theorem]{Lemma}

\newtheorem*{main-theorem}{Main Theorem}

\theoremstyle{definition}

\newtheorem{remark}[theorem]{Remark}
\newtheorem{example}[theorem]{Example}
\numberwithin{equation}{section}

\newcommand{\cc}{\mathbb{C}}

\newcommand{\nn}{\mathbb{N}}
\newcommand{\pp}{\mathbb{P}}
\newcommand{\qq}{\mathbb{Q}}
\newcommand{\rr}{\mathbb{R}}
\newcommand{\zz}{\mathbb{Z}}

\newcommand{\gp}{\text{gp}}

\newcommand{\uu}{\mathcal{U}}

\providecommand\ldb{\llbracket}
\providecommand\rdb{\rrbracket}

\hypersetup{
	pdftitle={FD},
	colorlinks=true,
	linkcolor=blue,
	citecolor=cyan,
	urlcolor=wine
}

\keywords{monoid algebra, atomicity, atomic monoid, factorization, finite rank-monoid, maximal common divisors}

\subjclass[2020]{Primary: 13F15, 13A05, 20M25; Secondary: 06F05, 11Y05, 13G05}

\begin{document}
	
	\mbox{}
	\title{On the ascent of atomicity to monoid algebras}
	
	\author{Felix Gotti}
	\address{Department of Mathematics\\MIT\\Cambridge, MA 02139}
	\email{fgotti@mit.edu}
	
	\author{Henrick Rabinovitz}
	\address{Department of Mathematics\\MIT\\Cambridge, MA 02139}
	\email{hrab@mit.edu}
	
\date{\today}

\begin{abstract}	 
	 A commutative cancellative monoid is atomic if every non-invertible element factors into irreducibles (also called atoms), while an integral domain is atomic if its multiplicative monoid is atomic. Back in the eighties, Gilmer posed the question of whether the fact that a torsion-free monoid~$M$ and an integral domain $R$ are both atomic implies that the monoid algebra $R[M]$ of $M$ over $R$ is also atomic. In general this is not true, and the first negative answer to this question was given by Roitman in 1993: he constructed an atomic integral domain whose polynomial extension is not atomic. More recently, Coykendall and the first author constructed finite-rank torsion-free atomic monoids whose monoid algebras over certain finite fields are not atomic. Still, the ascent of atomicity from finite-rank torsion-free monoids to their corresponding monoid algebras over fields of characteristic zero is an open problem. Coykendall and the first author also constructed an infinite-rank torsion-free atomic monoid whose monoid algebras (over any integral domain) are not atomic. As the primary result of this paper, we construct a rank-one torsion-free atomic monoid whose monoid algebras (over any integral domain) are not atomic. To do so, we introduce and study a methodological construction inside the class of rank-one torsion-free monoids that we call lifting, which consists in embedding a given monoid into another monoid that is often more tractable from the arithmetic viewpoint. For instance, the embedding in the lifting construction preserves the ascending chain condition on principal ideals and the existence of maximal common divisors.
\end{abstract}
\medskip

\maketitle


\bigskip
\section{Introduction}
\label{sec:intro}
\smallskip

A monoid is atomic if every non-invertible element factors into atoms/irreducibles (monoids here are tacitly assumed to be cancellative and commutative), while an integral domain is atomic if its multiplicative monoid is atomic. It is well known and routine to argue that every monoid/domain that satisfies the ascending chain condition on principal ideals (ACCP) is atomic. Given a monoid $M$ and an integral domain $R$, we let $R[M]$ denote the monoid algebra of $M$ over $R$, that is, the commutative ring consisting of all polynomial expressions with exponents in $M$ and coefficients in $R$ (it is well known that $R[M]$ is an integral domain if and only if~$M$ is a torsion-free monoid). The ascent of atomicity from the pair $(M,R)$ to the monoid algebra $R[M]$ is a delicate subject, which was first brought to attention by Gilmer back in the eighties (see \cite[page~189]{rG84}).
\begin{question} \label{quest:Gilmer atomicity}
	Given a torsion-free monoid $M$ and an integral domain $R$, does the fact that both $M$ and~$R$ are atomic suffice to guarantee that $R[M]$ is atomic?
\end{question}
\noindent In general, it is known that Question~\ref{quest:Gilmer atomicity} has a negative answer. The search for counterexamples motivated by this question has resulted in the development of methodologies and the introduction of certain notions that have been used in other contexts. For instance, the work on maximal common divisors (MCDs) developed by Roitman in~\cite{mR00} was later modified and used in~\cite[Theorem~2.5]{EK18} to embed an integral domain into an idf-domain, while the existence of MCDs was used in \cite[Theorem~3.1]{GP23} to provide a positive answer for the ascent of atomicity to polynomial semidomains. In addition, the monoid algebra constructed in~\cite[Theorem~5.4]{CG19} by Coykendall and the first author was later used in~\cite[Example~4.10]{GZ23} as a counterexample of a different nature.
\smallskip

The primary purpose of this paper is to settle down a still open problem stemming from Question~\ref{quest:Gilmer atomicity}, namely, whether the property of being atomic ascends from finite-rank torsion-free monoids to their corresponding monoid algebras over integral domains (this problem has only been resolved when the ring of coefficients is a finite field of prime cardinality). As important as solving this open problem is the machinery we develop along the way. Here we provide a methodological construction inside the class consisting of all rank-one torsion-free monoids that we call the `lifting construction'. This construction consists in embedding a given monoid into another monoid, called a `lifting monoid', which is often more tractable from the arithmetic viewpoint. Besides being a crucial component used in the proof of our main theorem, the lifting construction can be used to give an answer to the recent question of whether there exists a rank-one torsion-free atomic monoid that is not strongly atomic \cite[Question~4.4]{GV23}. We should emphasize that a positive answer to the same question was previously given in~\cite{GLRRT24} using an \emph{ad~hoc} construction.
\smallskip

The Hilbert-Basis-type question of whether the atomicity of an integral domain guarantees the atomicity of its polynomial extension is a specialization of Question~\ref{quest:Gilmer atomicity}, which was posed by Anderson, Anderson, and Zafrullah \cite[Question~1]{AAZ90} back in 1990. The first (negative) answer was given by Roitman in 1993: in the corresponding paper, he developed certain machinery on extensions of Rees algebras and existence of maximal common divisors that allowed him to construct two examples of atomic integral domains whose polynomial extensions are not atomic \cite[Examples~5.1 and~5.2]{mR93}. As of now, it seems that these are the only known examples of atomic integral domains whose polynomial extensions are not atomic (the integral domains constructed by Grams \cite[Theorem~1.3]{aG74} are perhaps the most tractable examples of atomic integral domains that do not satisfy the ACCP; however, their polynomial extensions have been proved to be atomic by Li and the first author \cite[Proposition~3.6]{GL22}). The ascent of atomicity from integral domains to their corresponding power series extensions was also settled down by Roitman: he constructed in~\cite[Example~1.4]{mR00} an atomic integral domain whose power series extension is not atomic.
\smallskip

In terms of the general problem encapsulated in Question~\ref{quest:Gilmer atomicity}, the dual question of whether the polynomial extension of an atomic integral domain is also atomic is that of whether the monoid algebra $F[M]$ of an atomic monoid~$M$ is also atomic over some/any field $F$. Unlike its dual, which was settled back in 1993, this latter question was not settled until recently: Coykendall and the first author constructed two examples of finite-rank torsion-free atomic monoids whose monoid algebras are not atomic over certain finite fields \cite[Theorems~4.2 and~5.4]{CG19}. In the same paper, the authors established the following result.

\begin{theorem} \label{thm:CG Universal Monoid}
	There exists a reduced infinite-rank torsion-free atomic monoid $M$ such that the monoid algebra $R[M]$ is not atomic for any integral domain~$R$.
\end{theorem}

\noindent Over fields of characteristic zero, it is still an open problem whether the property of being atomic ascends from finite-rank torsion-free monoids to their corresponding monoid algebras. A negative answer to this open problem is given by the following theorem, which is the main result we establish in this paper and an improvement of Theorem~\ref{thm:CG Universal Monoid}.
\smallskip

\begin{main-theorem}
	There exists a rank-one torsion-free atomic monoid $M$ such that the monoid algebra $R[M]$ is not atomic for any integral domain~$R$.
\end{main-theorem}
\smallskip

In Section~\ref{sec:background}, we revise the relevant notions and terminology we will use throughout this paper.
\smallskip

In Section~\ref{sec:lifting construction}, we introduce the lifting construction, which is the most important construction we will use later to prove our main theorem. A lifting monoid depends not only on the initial rank-one torsion-free monoid but also on the choice of certain function whose domain is a prescribed subset of the initial monoid. The choice of this function, which we call a `lifting function', plays a central role in the arithmetic properties of the final lifting monoid. The lifting construction is significantly motivated by the rank-one torsion-free monoid in Grams' construction of the first atomic integral domain not satisfying the ACCP. As it is the case for the monoid in Grams' construction, inside the lifting monoids the elements have certain canonical sum decompositions, and this is the most crucial property making our lifting construction so convenient. We will prove the existence and uniqueness of such a canonical decomposition right after describing the notion of lifting. Throughout the rest of the same section, we establish some needed facts about the lifting construction in terms of atomicity and principal ideals: we describe the set of atoms of the lifting monoid, argue that it is atomic if the initial monoid is atomic, and show that the ACCP is preserved by the lifting construction.
\smallskip

In Section~\ref{sec:ascent of atomicity}, we primarily focus on establishing our main theorem. Once we prove at the beginning of the section that the existence of certain maximal common divisors is preserved by the lifting construction, our proof of the main theorem can be naturally divided into the following three fundamental parts. First, we construct a rank-one torsion-free monoid that contains two elements without a maximal common divisor. Then we carefully choose a lifting function to lift the constructed monoid to a rank-one torsion-free atomic monoid that also contains two elements without a maximal common divisor. Finally, we use the properties of the resulting lifting monoid to argue that its monoid algebras are not atomic (over any integral domains). An integral domain is called Furstenberg if every nonzero nonunit is divisible by an irreducible (it follows directly from the definition that every atomic integral domain is Furstenberg). We conclude Section~\ref{sec:ascent of atomicity} showing that if the rings of coefficients of the monoid algebras constructed in the proof of our main theorem are taken to be fields, then such monoid algebras are Furstenberg domains. As a consequence, we obtain the first known examples of one-dimensional Furstenberg integral domains that are not atomic.

\bigskip
\section{Background}
\label{sec:background}

\smallskip
\subsection{General Notation}
\smallskip

Following usual conventions, we let $\zz$, $\qq$, and $\rr$ denote the set of integers, rational numbers, and real numbers, respectively. We let $\nn$ and $\nn_0$ denote the set of positive and nonnegative integers, respectively. In addition, we let $\pp$ denote the set of primes. For $b,c \in \zz$ with $b \le c$, we let $\ldb b, c \rdb$ denote the set of integers between $b$ and $c$; that is, $\ldb b,c \rdb := \{m \in \zz : b \le m \le c\}$. Also, for $S \subseteq \rr$ and $r \in \rr$, we set $S_{\ge r} := \{s \in S : s \ge r\}$, and we define $S_{> r}$, $S_{\le r}$ and $S_{< r}$ in a similar manner. We use $\sqcup$ instead of $\cup$ to emphasize when we are taking disjoint union of sets. For $q \in \qq_{>0}$,  the unique relatively prime positive integers $n$ and $d$ satisfying that $q = \frac nd$ are denoted here by $\mathsf{n}(q)$ and $\mathsf{d}(q)$, respectively. For each $p \in \pp$, we let $v_p \colon \qq^\times \to \zz$ denote the $p$-adic valuation map.

\medskip
\subsection{Commutative Monoids}

A \emph{monoid} is a semigroup with an identity element. However, in the context of this paper we tacitly assume that every monoid we deal with is both cancellative and commutative. Let $M$ be a monoid written additively. We let $M^\bullet$ denote the set of nonzero elements of $M$, and we call~$M$ \emph{trivial} provided that $M^\bullet$ is empty. The invertible elements of $M$, also called \emph{units}, form a subgroup of $M$, which we denote by $\uu(M)$. The monoid $M$ is called \emph{reduced} if the group $\uu(M)$ is trivial. The \emph{difference group} of $M$ (often called the \emph{Grothendieck group} of $M$) is the unique abelian group $\gp(M)$ up to isomorphism satisfying that any abelian group containing an isomorphic image of $M$ also contains an isomorphic image of $\gp(M)$. The monoid $M$ is called \emph{torsion-free} if $\gp(M)$ is a torsion-free group. The \emph{rank} of the monoid $M$, denoted by $\text{rank} \, M$, is the rank of $\gp(M)$ as a $\zz$-module, that is, the dimension of the vector space $\qq \otimes_\zz \gp(M)$ over~$\qq$. It is well known that every rank-one torsion-free monoid is isomorphic to an additive submonoid of $\qq$ (see~\cite[Section~24]{lF70} and \cite[Theorem~3.12]{GGT21}). When these monoids are not nontrivial groups, they can be realized as additive monoids consisting of nonnegative rationals, in which case they are known as \emph{Puiseux monoids}. The atomic structure and the arithmetic of factorizations of Puiseux monoids and their monoid algebras have been actively investigated in recent years (see~\cite{CJMM24,GG24,GL23} and references therein). Puiseux monoids play a crucial role throughout this paper. For a subset $S$ of $M$, we let $\langle S \rangle$ denote the submonoid of $M$ \emph{generated} by~$S$, namely, the smallest submonoid of $M$ containing $S$. The monoid~$M$ is called \emph{finitely generated} if $M = \langle S \rangle$ for some finite subset $S$ of $M$.
\smallskip

For $b,c \in M$, we say that $c$ \emph{divides} $b$ (\emph{in} $M$) if there exists $d \in M$ such that $b = c + d$; in this case, we write $c \mid_M b$. An element $d \in M$ is a \emph{common divisor} of a nonempty subset $S$ of $M$ provided that $d \mid_M s$ for all $s \in S$. A \emph{maximal common divisor} (MCD) of a nonempty subset $S$ of $M$ is a common divisor $d$ of~$S$ such that for any other common divisor $d' \in M$ of $S$ the divisibility relation $d \mid_M d'$ implies that $d' - d \in \uu(M)$. Following Roitman~\cite{mR93}, we say that $M$ is a $k$-\emph{MCD monoid} for some $k \in \nn$ if every subset $S$ of $M$ with $|S| = k$ has an MCD. It follows directly from the definition that every monoid is a $1$-MCD monoid and also that, for each $k \in \nn$, every $(k+1)$-MCD monoid is a $k$-MCD monoid. For each $k \in \nn$, there are $k$-MCD monoids that are not $(k+1)$-MCD monoids (see~\cite[Example~5.2]{mR93}). The monoid $M$ is called an \emph{MCD monoid} if $M$ is a $k$-MCD monoid for every $k \in \nn$.
\smallskip

A non-invertible element $a \in M$ is called an \emph{atom} of $M$ if whenever $a = b+c$ for some $b,c \in M$, either $b \in \uu(M)$ or $c \in \uu(M)$. We let $\mathcal{A}(M)$ denote the set consisting of all the atoms of~$M$. If $\mathcal{A}(M)$ is empty, then $M$ is called \emph{antimatter}. If every non-invertible element of $M$ is divisible by an atom, then $M$ is called a \emph{Furstenberg monoid}. An element of~$M$ is called \emph{atomic} if it is invertible or it can be written as a sum of finitely many atoms (allowing repetitions). Following Cohn~\cite{pC68}, we say that the monoid $M$ is \emph{atomic} if every element of $M$ is atomic. It follows directly from the definitions that every atomic monoid is a Furstenberg monoid. Following Anderson, Anderson, and Zafrullah~\cite{AAZ90}, we say that $M$ is \emph{strongly atomic} if $M$ is an atomic $2$-MCD monoid. 
\smallskip

A subset $I$ of $M$ is called an \emph{ideal} of~$M$ provided that the set $I + M := \{b+c : b \in I \text{ and } c \in M\}$ is contained in $I$ or, equivalently, $I+M = I$. An ideal of the form $b + M := \{b + c : c \in M\}$ for some $b \in M$ is called \emph{principal}. A sequence of principal ideals of $M$ is said \emph{to start} at the element $b \in M$ if the first term of the sequence is the principal ideal $b + M$. Also, an element $b \in M$ is said to satisfy the \emph{ascending chain condition on principal ideals} (ACCP) if every ascending chain of principal ideals starting at $b$ eventually stabilizes. The monoid $M$ is said to satisfy the {ACCP} if every element of $M$ satisfies the ACCP. It is well known and routine to verify that every monoid that satisfies the ACCP is atomic (see \cite[Proposition~1.1.4]{GH06}). The converse of this statement does not hold in general (see Example~\ref{ex:Grams monoid}).
\smallskip

Now assume that the monoid $M$ is atomic. One can readily check that the quotient $M/\uu(M)$ is also an atomic monoid. Let $\mathsf{Z}(M)$ denote the free (commutative) monoid on the set $\mathcal{A}(M/\uu(M))$, and let $\pi \colon \mathsf{Z}(M) \to M/\uu(M)$ be the unique monoid homomorphism that fixes every element of the set $\mathcal{A}(M/\uu(M))$. For every $b \in M$, we set $\mathsf{Z}(b) := \mathsf{Z}_M(b) := \pi^{-1} (b + \uu(M))$, and we call the elements of $\mathsf{Z}(b)$ \emph{additive factorizations} or, simply, \emph{factorizations} of~$b$. A recent survey on factorization theory in commutative monoids by Geroldinger and Zhong can be found in~\cite{GZ20}.

\medskip
\subsection{Monoid Algebras} 

Let $R$ be a commutative ring with identity. We let $R^*$ denote the multiplicative monoid of~$R$, that is, the monoid consisting of all non-zero-divisors of~$R$. When $R$ is an integral domain, we say that $R$ is \emph{Furstenberg} (resp., \emph{atomic}, \emph{strongly atomic}) provided that the monoid $R^*$ is Furstenberg (resp., atomic, strongly atomic). We let $\dim R$ denote the Krull dimension of $R$. 
\smallskip

Let $M$ be a monoid. We let $R[X;M]$ denote the monoid algebra of $M$ over $R$; that is, $R[X;M]$ is the commutative ring consisting of all polynomial expressions with exponents in $M$ and coefficients in~$R$. Following Gilmer's notation~\cite{rG84}, we write $R[M]$ instead of $R[X;M]$. Let $f(X)$ be a nonzero polynomial expression in $R[M]$. Then we can write $f(X)$ as follows:
\begin{equation} \label{eq:generic element in a monoid algebra}
	f(X) = r_1 X^{q_1} + \dots + r_k X^{q_k}
\end{equation}
for some nonzero coefficients $r_1, \dots, r_k \in R$ and some pairwise distinct exponents $q_1, \dots, q_k \in M$. The set $\text{supp} f(X) := \{q_1, \dots, q_k\}$ is uniquely determined by $f(X)$ and called the \emph{support} of $f(X)$. When~$M$ has a total order $\preceq$, we can assume that $q_1 \prec \dots \prec q_k$: in this case, the representation on the right-hand side of~\eqref{eq:generic element in a monoid algebra} is unique (as it is the case for standard polynomials), and the elements $\text{ord} \, f := q_1$ and $\deg f := q_k$ of the monoid $M$ are called the \emph{order} and the \emph{degree} of $f(X)$, respectively.
\smallskip

It is well known that the monoid algebra $R[M]$ is an integral domain if and only if $R$ is an integral domain and $M$ is a torsion-free monoid \cite[Theorem~8.1]{rG84}. When $R[M]$ is an integral domain, it follows from~\cite[Proposition 8.3]{GP74} that $\dim R[M] \ge 1 + \dim R$. In addition, when $R$ is Noetherian, it follows from \cite[Corollary~2]{jO88} that $\dim R[M] = \dim R + \text{rank} \, M$. As a result, the one-dimensional monoid algebras that are integral domains are precisely the monoid algebras of rank-one torsion-free monoids over fields. Background information on monoid algebras $R[M]$, emphasizing on the ascent of algebraic properties from the pair $(M, R)$ to $R[M]$ and including the most significant progress on the same subject until 1984, can be found in Gilmer's book~\cite{rG84}.

\bigskip
\section{The Lifting Construction}
\label{sec:lifting construction}

Additive submonoids of $\qq_{\ge 0}$ that are sparing play a central role in this section. We say that a subset~$S$ of $\qq_{> 0}$ is \emph{sparing} if there exist infinitely many $p \in \pp$ such that $v_p(q) \ge 0$ for all $q \in S^\bullet$, and we say that $p \in \pp$ is \emph{spared by} $S$ provided that $v_p(q) \ge 0$ for all $q \in S^\bullet$. Observe that every finite subset of $\qq_{> 0}$ is sparing and that a subset of $\qq_{> 0}$ is sparing if and only if the submonoid of $\qq$ it generates is sparing. In this section we introduce a construction that allows us to obtain from a sparing monoid $M$ another submonoid of $\qq_{\ge 0}$ that not only preserves some properties of $M$ but also earns new desirable properties: we call the monoid obtained from this construction a lifting monoid of $M$.

\medskip
\subsection{Descriptions of the Lifting Construction and the Lifting Decompositions}

We proceed to describe the lifting construction. For this purpose, we let $\mathcal{N}$ denote the class of all additive submonoids of $\nn_0$ (often called numerical monoids). Now let $M$ be a nontrivial sparing submonoid of $\qq_{\ge 0}$. Given a nonempty subset $S$ of $M^\bullet$ and an injective function $\pi \colon S \to \pp$, a \emph{lifting function} on $M$ is a function $\varphi \colon S \to \pp \times \mathcal{N}$ defined by the assignments $s \mapsto (\pi(s), N_s)$ that satisfies the following conditions:
\begin{enumerate}
	\item $v_{\pi(s)}(q) \ge 0$ and $v_{\pi(s)}(s) = 0$ for all $s \in S$ and $q \in M^\bullet$, and
	\smallskip
	
	\item $\pi(s) \in N_s$ for all $s \in S$.
\end{enumerate}
Since $M$ is sparing, such a function $\pi$ always exists for any prescribed nonempty subset $S$ of $M^\bullet$. Observe that the assignments $s \mapsto N_s$ (for any $s \in S$) are almost completely free except for the condition that $\pi(s) \in N_s$. Now, for each $q \in M$, set $M_q := \frac{q}{\pi(q)} N_q$ if $q \in S$ and $M_q := q \nn_0$ if $q \notin S$. For each $q \in M$, we see that~$M_q$ is a finitely generated Puiseux monoid containing~$q$. Finally, set
\[
	M_\varphi := \Big\langle \bigcup_{q \in M} M_q \Big\rangle,
\]
and call $M_\varphi$ the \emph{lifting} (\emph{monoid}) of $M$ with respect to~$\varphi$. Observe that $M \subseteq M_\varphi$ because $q \in M_q$ for every $q \in M^\bullet$. In what follows, we tacitly assume that $N_s \neq \pi(s) \nn_0$ for all $s \in S$. This assumption does not introduce any loss of generality as the restriction $\psi$ of~$\varphi$ to the set $\{s \in S : N_s \neq \nn_0 \pi(s) \}$ yields the same lifting; that is, $M_\psi = M_\varphi$. Note that there always exists a function $\pi$ such that the lifting monoid~$M_\varphi$ is also sparing for any choice of the lifting function $\varphi$.

Especial instances of the lifting construction have been used in the past to provide examples in commutative ring theory. For instance, the lifting monoid in the following example is the fundamental component in Grams' construction of the first atomic integral domain that does not satisfy the ACCP \cite[Theorem 1.3]{aG74} (in addition, see \cite[Theorem~3.3]{GL23} for a recent generalization of Grams' construction).

\begin{example} \label{ex:Grams monoid}
	Consider the additive monoid $M := \big\langle \frac 1{2^n} : n \in \nn_0 \big\rangle$, and set $S := \big\{ \frac 1{2^n} : n \in \nn_0 \big\} \subset M$. Let $(p_n)_{n \ge 0}$ be the strictly increasing sequence whose underlying set is $\pp \setminus \{2\}$, and then define the lifting function $\varphi \colon S \to \pp \times \mathcal{N}$ by setting $\varphi\big(\frac1{2^n}\big) = (p_n, \nn_0)$. In this case, $M_{\frac1{2^n}} = \frac{1}{2^n p_n} \nn_0$, and so the lifting monoid of $M$ with respect to $\varphi$ is
	\[
		M_\varphi = \Big\langle \frac 1{2^n p_n} : n \in \nn_0 \Big\rangle,
	\]
	which is called \emph{Grams' monoid}. It is not difficult to verify that $\mathcal{A}(M_\varphi) = \big\{ \frac1{2^n p_n} : n \in \nn_0 \big\}$ and so that the monoid $M_\varphi$ is atomic (see \cite[Example~3.2]{GL23}). However, $M_\varphi$ does not satisfy the ACCP because the ascending chain of principal ideals $\big( \frac1{2^n} + M_\varphi \big)_{n \ge 0}$ does not stabilize.
\end{example}

The fundamental advantage of the lifting construction is that in the obtained lifting monoid every element has a canonical sum decomposition.

\begin{prop} \label{prop:unique decomposition}
	Let $M$ be a sparing monoid, and let $\varphi \colon S \to \pp \times \mathcal{N}$ be a lifting function on $M$. Then any element $x \in M_\varphi$ can be uniquely decomposed as follows:
	\begin{equation} \label{eq:canonical decomposition}
		x = x_0 + \sum_{s \in S} x_s
	\end{equation}
with $x_0 \in M$ and $x_s \in M_s$ such that $s \nmid_{M_s} x_s$ for any $s \in S$.
\end{prop}

\begin{proof}
	Fix $x \in M_\varphi$. Since $M_\varphi$ is generated by the set $\bigcup_{q \in M} M_q$, one can write $x = \sum_{q \in M} y_q$, where $y_q \in M_q$ for every $q \in M$ and $y_q = 0$ for all but finitely many $q \in M$. Now fix $s \in S$, let $m_s$ be the maximum nonnegative integer such that $m_s s \mid_{M_s} y_s$, and set $ x_s := y_s - m_s s$. Note that if $s \nmid_{M_s} y_s$, then $m_s = 0$: in particular, $m_s = 0$ when $y_s = 0$. Observe that $x_s \in M_s$, and it follows from the maximality of $m_s$ that $s \nmid_{M_s} x_s$. After doing the same for each $s \in S$, we can write
	\[
		x = \sum_{q \in M \setminus S} y_q + \sum_{s \in S} y_s = \bigg( \sum_{q \in M \setminus S} y_q + \sum_{s \in S} m_s s \bigg) + \sum_{s \in S} x_s.
	\]
	Now set $x_0 :=  \sum_{q \in M \setminus S} y_q + \sum_{s \in S} m_s s$. Because $x_0 \in M$, the identity $x = x_0 + \sum_{s \in S} x_s$ is the desired sum decomposition specified in~\eqref{eq:canonical decomposition}.
	
	For the uniqueness, let $x = x'_0 + \sum_{s \in S} x'_s$ be a sum decomposition as in~\eqref{eq:canonical decomposition} (satisfying the same properties), and let us show that both sum decompositions of $x$ are the same. Assume that there exists $t \in S$ such that either $x_t \neq 0$ or $x'_t \neq 0$ (otherwise, both sum decompositions are the same). Write $x_t = \frac{t}{\pi(t)} n_t$ and $x'_t = \frac{t}{\pi(t)} n'_t$ for some $n_t, n'_t \in N_t$. Assume, without loss of generality, that $x'_t \ge x_t$ and, therefore, that $n'_t \ge n_t$. We can apply the $\pi(t)$-adic valuation map on both sides of the equality
	\[
		x'_t - x_t = (x_0 - x'_0) + \sum_{s \in S \setminus \{t\}} (x_s - x'_s)
	\]
	to obtain that $v_{\pi(t)}(x'_t - x_t) \ge 0$ (we are using here that $\pi$ is injective). Since $v_{\pi(t)}(t) = 0$, the fact that $v_{\pi(t)}\big( \frac{t}{\pi(t)} (n'_t - n_t)\big) = v_{\pi(t)}(x'_t - x_t) \ge 0$ implies that $\pi(t) \mid n'_t - n_t$. As a consequence, the equality $x'_t = \frac{n'_t - n_t}{\pi(t)} t + x_t$, along with the fact that $t \nmid_{M_t} x'_t$, ensures that $n'_t - n_t = 0$, which implies that $x'_t = x_t$. Thus, $x'_s = x_s$ for all $s \in S$ and, therefore, $x'_0 = x_0$. Hence the uniqueness follows.
\end{proof}

Assume the notation in the statement of Proposition~\ref{prop:unique decomposition}. We call $x_0$ the $M$-\emph{projection} of $x$ while, for each $s \in S$, we call $x_s$ the $M_s$-\emph{projection} of $x$. In addition, for each $s \in S$, we call any $p_s \in M_s$ such that $s \nmid_{M_s} p_s$ an $M_s$-\emph{projection}. Finally, we call (the right-hand side of) the equality~\eqref{eq:canonical decomposition} the $\varphi$-\emph{lifting decomposition} of $x$.

\medskip
\subsection{Properties of the Lifting Construction}

We proceed to establish some helpful properties of the lifting monoid and the lifting decomposition.

\begin{prop} \label{prop:relation of M-projections}
	Let $M$ be a sparing monoid, and let $\varphi \colon S \to \pp \times \mathcal{N}$ be a lifting function on $M$. Then the following statements hold.
	\begin{enumerate}
		\item For each $s \in S$, an element $p_s \in M_\varphi$ is an $M_s$-projection if and only if both conditions $p_s \in M_s$ and $\pi(s) \nmid_{N_s} \frac{\pi(s)}{s} p_s$ hold.
		\smallskip
		
		\item For each $M_s$-projection $p_s$, there exists a unique $M_s$-projection $q_s$ such that $p_s + q_s \in \nn_0 s$.
	\end{enumerate}
\end{prop}

\begin{proof}
	(1) Fix $s \in S$, and take $p_s \in M_\varphi$. By definition, $p_s$ is an $M_s$-projection if and only if $p_s \in M_s$ and $s \nmid_{M_s} p_s$, which happens if and only if $p_s \in M_s \setminus \{ s + \frac{s}{\pi(s)} n : n \in N_s \}$. We are done because the last statement is equivalent to the fact that $p_s \in M_s$ and $\frac{\pi(s)}s p_s \notin \pi(s) + N_s$.
	\smallskip
	
	(2) Fix $s \in S$, and let $p_s$ be an $M_s$-projection. For the existence, write $p_s = \frac{s}{\pi(s)} n$ for some $n \in N_s$, and then take the minimum $m \in N_s$ such that $\pi(s) \mid n+m$. Now set $q_s := \frac{s}{\pi(s)} m \in M_s$. The minimality of $m$ ensures that $\pi(s) \nmid_{N_s} m$, and so it follows from part~(1) that $q_s$ is an $M_s$-projection. Finally, it is clear that $p_s + q_s \in \nn_0s$. For the uniqueness, suppose that $q_s$ and $q'_s$ are two $M_s$-projections such that $p_s + q_s$ and $p_s + q'_s$ both belong to $\nn_0 s$, and assume that $q_s \le q'_s$. Then we see that $q'_s - q_s \in \nn_0 s$, and so the fact that $s \nmid_{M_s} q'_s$ implies that $q'_s = q_s$.
\end{proof}

We now verify that any divisibility relation in the lifting monoid enforces a divisibility relation between the corresponding $M$-projections as well as certain divisibility relations between the corresponding $M_s$-projections.

\begin{prop} \label{prop:divisibility relation of M-projections}
	Let $M$ be a sparing monoid, and let $\varphi \colon S \to \pp \times \mathcal{N}$ be a lifting function on $M$. For $b,c \in M$, let $b = b_0 + \sum_{s \in S} b_s$ and $c = c_0 + \sum_{s \in S} c_s$ be the $\varphi$-lifting decompositions of $b$ and $c$, respectively. If $b \mid_{M\varphi} c$, then the following statements hold.
	\begin{enumerate}
		\item $b_0 \mid_M c_0$.
		\smallskip
		
		\item If $b_s > c_s$ for some $s \in S$, then $b_0 + s \mid_M c_0$.
	\end{enumerate}
\end{prop}

\begin{proof}
	Take $b' \in M_\varphi$ such that $c = b + b'$. Now let $b' = b'_0 + \sum_{s \in S} b'_s$ be the $\varphi$-lifting decomposition of~$b'$.
	\smallskip
	
	(1) We can obtain a $\varphi$-lifting decomposition of $c$ from the identity $c = b_0 + b'_0 + \sum_{s \in S} (b_s + b'_s)$ by replacing each summand $b_s + b'_s$ by its $M_s$-projection and, to compensate the equality, adding $m := \sum_{s \in S} m_s s$ to $b_0 + b'_0$, where $m_s$ is the maximum nonnegative integer such that $m_s s \mid_{M_s} b_s + b'_s$ (observe that $m_s = 0$ for all but finitely many $s \in S$). It is clear that $m \in M$. Hence the uniqueness of the $\varphi$-lifting decomposition guarantees that $c_0 = b_0 + b'_0 + m$. As $m \in M$, it follows that $b_0 \mid_M c_0$.
	\smallskip
	
	(2) For each $s \in S$, let $m_s$ be defined as in the previous part, and also set $m := \sum_{s \in S} m_s s \in M$. As we have seen before, $c_0 = b_0 + b'_0 + m$. If for some $s \in S$ the inequality $b_s > c_s$ holds, then $b_s + b'_s > c_s$ and, therefore, the uniqueness of the $\varphi$-lifting decomposition of $c$ guarantees that $m_s \ge 1$: this in turn implies that $s \mid_M m$, whence from $c_0 = b_0 + b'_0 + m$ we infer that $b_0 + s \mid_M c_0$.
\end{proof}

We now consider atomicity under the lifting construction: in particular, we will see that the property of being atomic ascends from any sparing monoid to its lifting monoids.

\begin{prop} \label{prop:atoms in lifting constructions}
	Let $M$ be a sparing monoid, and let $\varphi \colon S \to \pp \times \mathcal{N}$ be a lifting function on $M$. Then the following statements hold.
	\begin{enumerate}
		\item $\mathcal{A}(M) \setminus S \subseteq \mathcal{A}(M_\varphi)$.
		\smallskip

		\item $\mathcal{A}(M_q) \setminus \{q\} \subseteq \mathcal{A}(M_\varphi)$ for every $q \in M$.
		\smallskip
		
		\item $\mathcal{A}(M_\varphi) = \big( \mathcal{A}(M) \setminus S \big) \sqcup \{a \in \mathcal{A}(M) \cap S : a \in \mathcal{A}(M_a)\} \sqcup \big( \bigcup_{s \in S} \mathcal{A}(M_s) \setminus \{s\} \big)$.
		\smallskip
		
		\item If $M$ is atomic, then $M_\varphi$ is atomic.
	\end{enumerate}
\end{prop}

\begin{proof}
	(1) Take $a \in \mathcal{A}(M) \setminus S$. To argue that $a$ is also an atom of $M_\varphi$, write $a = x+y$ for some $x,y \in M_\varphi$. Let $x = x_0 + \sum_{s \in S} x_s$ and $y = y_0 + \sum_{s \in S} y_s$ be the $\varphi$-lifting decompositions of $x$ and~$y$, respectively. If $x_s \neq 0$ for some $s \in S$, then it would follow from part~(2) of Proposition~\ref{prop:divisibility relation of M-projections} that $x_0 + s \mid_M a$, and so the fact that $a \in \mathcal{A}(M)$ would imply that $a \in \{x_0, s\}$, which is not possible because $s \notin \{0,a\}$. Hence $x_s = 0$ for every $s \in S$. In a similar way, one can infer that $y_s = 0$ for every $s \in S$. Now the identity $a = x_0 + y_0$, along with the fact that $a$ is an atom of $M$, guarantees that either $x_0 = 0$ or $y_0 = 0$, which in turn implies that $x=0$ or $y=0$. Thus, $a \in \mathcal{A}(M_\varphi)$, and we can conclude that $\mathcal{A}(M) \setminus S \subseteq \mathcal{A}(M_\varphi)$.
	\smallskip
	
	(2) Fix $q \in M$. If $q \notin S$, then $M_q = \nn_0 q$ and, therefore, $\mathcal{A}(M_q) \setminus \{q\} = \emptyset \subseteq \mathcal{A}(M_\varphi)$. Thus, we assume that $q \in S$. Take $a \in \mathcal{A}(M_q) \setminus \{q\}$. Since $a \neq q$ and $q \neq 0$, it follows that $q \nmid_{M_q} a$. This means that $a$ is its own $M_q$-projection. To argue that $a$ is an atom of $M_\varphi$, write $a = b + c$ for some $b,c \in M_\varphi$, and let $b = b_0 + \sum_{s \in S} b_s$ and $c = c_0 + \sum_{s \in S} c_s$ be the $\varphi$-lifting decompositions of $b$ and $c$, respectively. Observe that the $\varphi$-lifting decomposition of $a$ can be obtained from the identity $a = b_0 + c_0 + \sum_{s \in S} (b_s + c_s)$ by replacing each $b_s + c_s$ by its $M_s$-projection and, to compensate the equality, adding $m := \sum_{s \in S} m_s s$ to $b_0 + c_0$, where $m_s$ is the maximum nonnegative integer such that $m_s s \mid_{M_s} b_s + c_s$ (so $m_s = 0$ for all but finitely many $s \in S$). Since $a$ is its own $M_q$-projection, the $M$-projection $b_0 + c_0 + m$ of $a$ equals~$0$, which implies that $b_0 = c_0 = m = 0$. This in turn implies that, for each $s \in S$, the equality $m_s = 0$ holds, and so $b_s + c_s$ is the $M_s$-projection of $a$. Thus, $b_s = c_s = 0$ for any $s \in S \setminus \{q\}$ and $a = b_q + c_q$. Since $a \in \mathcal{A}(M_q)$, either $b_q = 0$ or $c_q = 0$, which implies that either $b = 0$ or $c = 0$. Hence $a \in \mathcal{A}(M_\varphi)$, and so the inclusion $\mathcal{A}(M_q) \setminus \{q\} \subseteq \mathcal{A}(M_\varphi)$ holds.
	\smallskip
	
	(3) Set $A_S := \{a \in \mathcal{A}(M) \cap S : a \in \mathcal{A}(M_a)\}$ and $U := \bigcup_{s \in S} \mathcal{A}(M_s) \setminus \{s\}$. We want to verify that $\mathcal{A}(M_\varphi) = (\mathcal{A}(M) \setminus S) \sqcup A_S \sqcup U$. To do so, take $a \in \mathcal{A}(M_\varphi)$. Assume first that $a \in M$. In this case, $a \in \mathcal{A}(M)$. If $a \notin S$, then $a \in \mathcal{A}(M) \setminus S$. If $a \in S$, then the containment $a \in \mathcal{A}(M_a)$ ensures that $a \in A_S$. Now assume that $a \notin M$. Because $a \in \mathcal{A}(M_\varphi)$, it follows from Proposition~\ref{prop:unique decomposition} that $a \in M_s$ for some $s \in S$, and so $a \in \mathcal{A}(M_s)$. Since $a \notin M$, it follows that $a \in \mathcal{A}(M_s) \setminus \{s\}$, which implies that $a \in U$. Hence $\mathcal{A}(M_\varphi) \subseteq (\mathcal{A}(M) \setminus S) \sqcup A_S \sqcup U$  as the three sets involved in the union are clearly pairwise disjoint.
	
	 Now let us argue the reverse inclusion. It follows from part~(1) that $\mathcal{A}(M) \setminus S \subseteq \mathcal{A}(M_\varphi)$. To argue that $A_S \subseteq \mathcal{A}(M_\varphi)$, take $a \in A_S$ and write $a = q + r$ for some $q,r \in M_\varphi$. Assume first that $q \notin M$. Thus, for some $s \in S$, the $M_s$-projection in the $\varphi$-lifting decomposition of~$q$ is nonzero, and so it follows from part~(2) of Proposition~\ref{prop:divisibility relation of M-projections} that $s \mid_M a$. This, along with the fact that $a \in \mathcal{A}(M)$, ensures that $a = s$ and, therefore, $q = s$ and $r=0$. We can also conclude that $q=0$ if we assume that $r \notin M$. Lastly, if $q,r \in M$, then one of them must equal $0$ because $a \in \mathcal{A}(M)$. Hence $a \in \mathcal{A}(M_\varphi)$, and so the inclusion $A_S \subseteq \mathcal{A}(M_\varphi)$ holds. Now observe that the inclusion $U \subseteq \mathcal{A}(M_\varphi)$ also holds by part~(2). Hence $(\mathcal{A}(M) \setminus S) \sqcup A_S \sqcup U \subseteq \mathcal{A}(M_\varphi)$.
	\smallskip
	
	(4) Assume that $M$ is atomic. First, let us argue that every element of $\mathcal{A}(M)$ is atomic in $M_\varphi$. Fix $a \in \mathcal{A}(M)$. If $a \notin S$, then it follows from part~(1) that $a$ belongs to $\mathcal{A}(M_\varphi)$, and so $a$ is an atomic element of $M_\varphi$. Suppose, therefore, that $a \in S$. We split the rest of our argument into the following two cases.
	\smallskip
	
	\noindent \textsc{Case 1:} $a \notin \mathcal{A}(M_a)$. In this case, since $M_a$ is atomic (because it is finitely generated), $a$ can be written as a sum of elements in the set $\mathcal{A}(M_a)$, and the fact that $a \notin \mathcal{A}(M_a)$ ensures that the elements in such a sum are indeed in $\mathcal{A}(M_a) \setminus \{a\}$. Thus, the inclusion in part~(2) guarantees that $a$ is atomic in $M_\varphi$.
	\smallskip
	
	\noindent \textsc{Case 2:} $a \in \mathcal{A}(M_a)$. If $a \in \mathcal{A}(M_\varphi)$, then $a$ is an atomic element of $M_\varphi$. Therefore we will assume that~$a \notin \mathcal{A}(M_\varphi)$. Write $a = b+c$ for some $b,c \in M_\varphi^\bullet$. As $a \in M$, the $M_s$-projection of~$a$ equals zero for all $s \in S$. Hence, since $a \in \mathcal{A}(M)$, it follows from part~(2) of Proposition~\ref{prop:divisibility relation of M-projections} that for each $s \in S \setminus \{a\}$ the $M_s$-projection of $b$ equals zero. Thus, $b = b_0 + b_a$ is the $\varphi$-lifting decomposition of~$b$, where $b_0$ and~$b_a$ are the $M$-projection and the $M_a$-projection of $b$, respectively. Similarly, $c = c_0 + c_a$ is the $\varphi$-lifting decomposition of $c$, where $c_0$ and $c_a$ are the $M$-projection and the $M_a$-projection of $c$, respectively. As $a \in \mathcal{A}(M)$ and both $b$ and $c$ are positive, we see that $b_a + c_a > 0$. Since $a = (b_0 + c_0) + (b_a + c_a)$, it follows from the uniqueness of the $\varphi$-lifting decomposition that $b_a + c_a = na$ for some $n \in \nn$, which implies that $b_0 = c_0 = 0$ and $n=1$. Hence $a = b_a + c_a$, and from the fact that both $b$ and $c$ are positive we infer that $b_a$ and $c_a$ are both less than $a$. As a result, we can write both $b_a$ and $c_a$ as sums of elements in $\mathcal{A}(M_a) \setminus \{a\}$, and so it follows from part~(2) that $b_a$ and $c_a$ are atomic elements in $M_\varphi$. Hence $a$ is also an atomic element in $M_\varphi$.
	\smallskip
	
	Thus, we have argued that every element of $\mathcal{A}(M)$ is atomic in $M_\varphi$. Because $M$ is atomic, this implies that every element of $M$ is atomic in $M_\varphi$. In particular, every element of $S$ is atomic in $M_\varphi$. Now for each $s \in S$, every element of $M_s^\bullet$ can be written as the sum of copies of $s$ and copies of the elements in the set $\mathcal{A}(M_s) \setminus \{s\}$. Hence it follows from part~(2) that for any fixed $s \in S$, each element of $M_s$ is atomic in $M_\varphi$. As $M \cup \big(\bigcup_{s \in S} M_s \big)$ is a generating set of $M_\varphi$, we can conclude that every element of $M_\varphi$ is atomic; that is, $M_\varphi$ is an atomic monoid.
\end{proof}

The lifting construction is an alternative construction to that called atomization in \cite[Section~3]{GL23}, where the authors also embed a given rank-one torsion-free monoid into another by introducing prime denominators. As suggested by the term ``atomization", the monoid constructed in the atomization construction is always atomic (even if the initial monoid is not) \cite[Proposition~3.1]{GL23}. Unlike the monoid obtained by atomization, the monoid obtained as a lifting monoid of a sparing monoid may not be atomic. The following example illustrates this observation.

\begin{example}
	Let $M$ be the additive monoid consisting of all dyadic nonnegative rational numbers, namely, $\zz[\frac12]_{\ge 0}$. Clearly, $M$ is a sparing (Puiseux) monoid. Also observe that $M$ is antimatter. Let $(s_n)_{n \ge 1}$ be an enumeration of the elements of $M^\bullet$, and then let $(p_n)_{n \ge 1}$ be a strictly increasing sequence of odd primes such that $v_{p_n}(s_n) = 0$ and $p_{n+1} s_n > p_n$ for every $n \in \nn$. Now let $M_\varphi$ be the lifting monoid of $M$ with respect to the lifting map $\varphi \colon M^\bullet \to \pp \times \mathcal{N}$ given by $\varphi(s_n) := (p_n, \langle p_n, p_{n+1} \rangle)$ (here $\pi$ is the map given by the assignments $s_n \mapsto p_n$, which is injective because the sequence $(p_n)_{n \ge 1}$ is strictly increasing). Observe that, in this case, $M_{s_n} := \big\langle s_n, \frac{p_{n+1}}{p_n} s_n \big\rangle$ for every $n \in \nn$. Since $M$ is antimatter, no element of $M$ is an atom of $M_\varphi$. Therefore
	\[
		\mathcal{A}(M_\varphi) \subseteq \bigcup_{n \in \nn} \mathcal{A}(M_{s_n}) = \Big\{s_n, \frac{p_{n+1}}{p_n} s_n : n \in \nn \Big\}.
	\]
	From this, one can readily argue that $\mathcal{A}(M_\varphi) = \big\{ \frac{p_{n+1}}{p_n} s_n \mid n \in \nn \big\}$. Now the fact that $p_{n+1} s_n > p_n$ for every $n \in \nn$ guarantees that $\inf \mathcal{A}(M_\varphi) \ge 1$. As a consequence, no element of the set $M \cap (0,1)$ is atomic in $M_\varphi$. Thus, the lifting monoid $M_\varphi$ is not atomic.
\end{example}

As our next proposition indicates, both the lifting construction and the lifting decomposition have a desirable behavior with respect to satisfying the ACCP.

\begin{prop} \label{prop:ACCP in lifting constructions}
	Let $M$ be a sparing monoid, and let $\varphi \colon S \to \pp \times \mathcal{N}$ be a lifting function on $M$. Then the following statements hold.
	\begin{enumerate}
		\item An element of $M_\varphi$ satisfies the ACCP in $M_\varphi$ if and only if its $M$-projection satisfies the ACCP in $M$.
		\smallskip
		
		\item The lifting monoid $M_\varphi$ satisfies the ACCP if and only if $M$ satisfies the ACCP.
	\end{enumerate}
\end{prop}

\begin{proof}
	(1) Fix a nonzero $b \in M_\varphi$, and let $b_0$ be the $M$-projection of $b$. Let us show that $b$ satisfies the ACCP in $M_\varphi$ if and only if $b_0$ satisfies the ACCP in $M$.
	
	For the direct implication, suppose that $b$ satisfies the ACCP in $M_\varphi$. Observe that to each ascending chain of principal ideals $(b_n + M)_{n \ge 0}$ in~$M$ starting at $b_0$, we can assign the ascending chain of principal ideals $(b_n + M_\varphi)_{n \ge -1}$ in $M_\varphi$ starting at $b$, after setting $b_{-1} := b$ (the later sequence of ideals is ascending because $b_0 \mid_{M_\varphi} b$). Under this assignment, it follows from part~(1) of Proposition~\ref{prop:divisibility relation of M-projections} that $(b_n + M)_{n \ge 0}$ stabilizes if and only if $(b_n + M_\varphi)_{n \ge -1}$ stabilizes. As a result, the fact that $b$ satisfies the ACCP in $M_\varphi$ implies that $b_0$ satisfies the ACCP in $M$.
	
	Conversely, suppose that $b_0$ satisfies the ACCP in $M$. Now let $(b'_n + M_\varphi)_{n \ge 1}$ be an ascending chain of principal ideals in $M_\varphi$ starting at $b'_1 := b$. It follows from part~(1) of Proposition~\ref{prop:divisibility relation of M-projections} that the sequence of principal ideals of $M$ that we obtain after replacing, for each $k \in \nn$, the term $b'_k + M_\varphi$ of the sequence $(b'_n + M_\varphi)_{n \ge 1}$ by the principal ideal of $M$ generated by the $M$-projection of $b'_k$ is an ascending chain of principal ideals of $M$ starting at $b_0$, and so it must stabilize because $b_0$ satisfies the ACCP in $M$. Thus, from one point on all the terms of the sequence $(b'_n)_{n \ge 1}$ have the same $M$-projection, and so we can assume, without loss of generality, that all the terms of the sequence $(b'_n)_{n \ge 1}$ have the same $M$-projection. As a result, it follows from part~(2) of Proposition~\ref{prop:divisibility relation of M-projections} that, for each $s \in S$, the sequence $(p_s(b'_n))_{n \ge 1}$ obtained from replacing each term of the sequence $(b'_n)_{n \ge 1}$ by its $M_s$-projection $p_s(b'_n)$ must be a decreasing sequence and, because all the terms of $(p_s(b'_n))_{n \ge 1}$ belong to the finitely generated Puiseux monoid~$M_s$, the same sequence must stabilize. This, along with the fact that all but finitely many $M_s$-projections of $b'_1$ are zero, guarantees that the ascending chain of principal ideals $(b'_n + M_\varphi)_{n \ge 1}$ must also stabilize. Hence $b$ satisfies the ACCP in $M_\varphi$.
	\smallskip
	
	(2) This is an immediate consequence of part~(1).
\end{proof}

\bigskip
\section{Ascent of Atomicity to Monoid Algebras}
\label{sec:ascent of atomicity}

As mentioned in the introduction, it is still an open problem whether in the class of finite-rank torsion-free monoids, the property of being atomic ascends to monoid algebras over fields of characteristic zero. This problem will be solved as a consequence of our main theorem, which we will establish in this section. First, we need to argue that the lifting construction preserves being a $k$-MCD (for each $k \in \nn$).

\medskip
\subsection{The Lifting Construction and Maximal Common Divisors}

In this subsection we first show that, for each $k \in \nn$, the property of being a $k$-MCD monoid transfers between a rank-one torsion-free monoid and its lifting monoids. Then we use this result (the case $k=2$) to construct an example of an atomic Puiseux monoid that is not strongly atomic. A more refined  version of this preliminary construction plays a crucial role in the proof of our main theorem.

\begin{lemma} \label{lem:k-MCD transfer under lift construction}
	Let $M$ be a sparing monoid, and let $M_\varphi$ be the lifting monoid of $M$ with respect to a lifting map~$\varphi$. Then for each $k \in \nn$, the monoid $M$ is a $k$-MCD monoid if and only if $M_\varphi$ is a $k$-MCD monoid.
\end{lemma}

\begin{proof}
	Fix $k \in \nn$. It suffices to argue the statement of the lemma for $k \ge 2$. 
	
	For the direct implication, suppose that $M$ is a $k$-MCD monoid. To argue that $M_\varphi$ is also a $k$-MCD monoid, fix $b_1, b_2, \dots, b_k \in M_\varphi$. Let $d$ be a maximal common divisor in $M$ of the $M$-projections of $b_1, b_2, \dots, b_k$. Proving that $b_1, b_2, \dots, b_k$ have a maximal common divisor in $M_\varphi$ amounts to showing that $b_1-d, b_2-d, \dots, b_k-d$ have a maximal common divisor in $M_\varphi$. Clearly, the only common divisor of the $M$-projections of $b_1- d, b_2-d, \dots, b_k-d$ in $M$ is $0$. As a consequence, it follows from part~(1) of Proposition~\ref{prop:divisibility relation of M-projections} that the $M$-projection of any common divisor of $b_1- d, b_2-d, \dots, b_k-d$ in $M_\varphi$ must be~$0$. We are done once we argue that the set $C$ consisting of all common divisors of $b_1- d, b_2-d, \dots, b_k-d$ in $M_\varphi$ is finite. To do this, take $c \in C$, and let $c = \sum_{s \in S} c_s$ be the $\varphi$-lifting decomposition of $c$. Suppose that $c_s \neq 0$ for some $s \in S$. If for every $i \in \ldb 1,k \rdb$ the inequality $c_s > (b_i- d)_s$ holds, then it follows from part~(2) of Proposition~\ref{prop:divisibility relation of M-projections} that $s$ divides the $M$-projection of $b_i - d$ in $M$, which we have already observed that is not possible. Hence $c_s \le \max\{(b_i - d)_s : i \in \ldb 1,k \rdb \}$. Now let $T$ be the finite set consisting of all $s \in S$ such that $(b_i - d)_s > 0$ for some $i \in \ldb 1,k \rdb$. It is clear that
	\begin{align*} \label{eq:aux finite set I}
		C &\subseteq D := \bigg\{ \sum_{s \in S} d_s : d_s \in M_s, \ s \nmid_{M_s} d_s \text{ for all } s \in S,  \text{ and } d_s \le (b_i - d)_s  \text{ for some } i \in \ldb 1,k \rdb \bigg\}.
	\end{align*}
	We claim that $D$ is a finite set. To see this, take $\sum_{s \in S} d_s \in D$, where $d_s \in M_s$ and $s \nmid_{M_s} d_s$ for all $s \in S$. Then for each $s \in S \setminus T$, the fact that $(b_i - d)_s = 0$ for every $i \in \ldb 1,k \rdb$ implies that $d_s = 0$. As a result, $\sum_{s \in S} d_s = \sum_{s \in T} d_s$. Thus, the fact that $T$ is a finite set, along with the inequalities $d_s \le (b_i - d)_s$ for some $i \in \ldb 1,k \rdb$ (for all $s \in S$) in the definition of $D$, guarantees that $D$ is a finite set. This in turn implies that~$C$ is a finite set. As a result, the elements $b_1, b_2, \dots, b_k$ have a maximal common divisor in $M_\varphi$. Hence $M_\varphi$ is also a $k$-MCD monoid.
	\smallskip
	
	Conversely, suppose that $M_\varphi$ is a $k$-MCD monoid. Let us argue that the arbitrarily chosen elements $b_1, b_2, \dots, b_k \in M$ have a maximal common divisor in $M$. Let $d$ be a maximal common divisor of such elements in $M_\varphi$. Let $d = d_0 + \sum_{s \in S} d_s$ be the $\varphi$-lifting decomposition of $d$. Suppose, by way of contradiction, that $d_y > 0$ for some $y \in S$. By part~(2) of Proposition~\ref{prop:relation of M-projections}, there exists exactly one $M_y$-projection $p_y$ such that $d_y + p_y \in \nn_0 y$. Note that $p_y > 0$ because $d_y > 0$. Now fix $i \in \ldb 1,k \rdb$, and let $b_i - d = x_0 + \sum_{s \in S} x_s$ be the $\varphi$-lifting decomposition of $b_i - d$. Since the $M_y$-projection of $b_i$ is $0$, we see that $d_y + x_y \in \nn_0 y$, and so the uniqueness of $p_y$ guarantees that the $M_y$-projection of $b_i - d$ is~$p_y$. Therefore $p_y \mid_{M_{\varphi}} b_i - d$ for every $i \in \ldb 1,k \rdb$. However, this contradicts that $d$ is a maximal common divisor of $b_1, b_2, \dots, b_k$ in $M_\varphi$. As a consequence, $d_s = 0$ for any $s \in S$, which implies that $d \in M$. Then it follows from part~(1) of Proposition~\ref{prop:divisibility relation of M-projections} that $d$ is a common divisor of $b_1, \dots, b_k$ in $M$. Finally, observe that the only common divisor of $b_1 - d, b_2 - d, \dots, b_k - d$ in $M$ is $0$ because the only common divisor of the same elements in $M_\varphi$ is $0$. Hence $d$ is a maximal common divisor of $b_1, b_2, \dots, b_k$ in $M$. We conclude then that $M$ is also a $k$-MCD monoid.
\end{proof}

We proceed to construct a rank-one torsion-free monoid that will play a central role in the rest of this paper. Fix $\epsilon \in \rr_{>0}$ with $\epsilon < \frac1{10}$, and let $(a_n)_{n \ge 2}$ be a sequence consisting of positive rationals such that $(\mathsf{d}(a_n))_{n \ge 2}$ is a strictly increasing sequence of primes and $\sum_{n \ge 2} a_n < \frac\epsilon{8}$. Furthermore, let us assume that the underlying set $A_\epsilon$ of $(a_n)_{n \ge 2}$ is a sparing set. Let $b_1$ and $c_1$ be rationals in the interval  $(1-\frac\epsilon{8}, 1)$ with distinct prime denominators such that the primes $\mathsf{d}(b_1)$ and $\mathsf{d}(c_1)$ are both spared by $A_\epsilon$. In particular, $b_1, c_1 \notin \langle A_\epsilon \rangle$. Then let $(b_n)_{n \ge 1}$ and $(c_n)_{n \ge 1}$ be the sequences of rationals whose terms are respectively defined by the following equalities:
\begin{equation} \label{eq:def of the b_j's and c_j_s}
	b_n := b_1 - \sum_{k=2}^n a_k \quad \text{ and } \quad c_n := c_1 - \sum_{k=2}^n a_k
\end{equation}
for every $n \in \nn_{\ge 2}$. Observe that $b_n = b_{n+1} + a_{n+1}$ and $c_n = c_{n+1} + a_{n+1}$ for every $n \in \nn$. Now we let $B_\epsilon$ and $C_\epsilon$ denote the underlying sets of the sequences $(b_n)_{n \ge 1}$ and $(c_n)_{n \ge 1}$, respectively. Observe that the elements of both $B_\epsilon$ and $C_\epsilon$ belong to the interval $(1-\frac{\epsilon}4, 1)$. Lastly, let $M$ be the monoid generated by the set $A_\epsilon \cup B_\epsilon \cup C_\epsilon$; that is,
\begin{equation} \label{eq:main monoid}
	M := \big\langle A_\epsilon \cup B_\epsilon \cup C_\epsilon \big\rangle. 
\end{equation}
The monoid $M$ is a Puiseux monoid and, therefore, a rank-one torsion-free monoid. Throughout the rest of this section, we adopt the notation introduced in this paragraph.

\begin{lemma} \label{lem:a monoid that is not 2-MCD}
	Let the monoid $M$ and the sets $A_\epsilon, B_\epsilon$, and $C_\epsilon$ be as in~\eqref{eq:main monoid}. Then the following statements hold.
	\begin{enumerate}
		\item[(1)] Any nonempty finite subset of $B_\epsilon \cup C_\epsilon$ has a common divisor in $M$ that belongs to~$A_\epsilon$.
		\smallskip
		
		\item[(2)] The monoid $M$ is not a $2$-MCD monoid.
	\end{enumerate}
\end{lemma}

\begin{proof}
	(1) Let $S$ be  a nonempty finite subset of $B_\epsilon \cup C_\epsilon$, and then set
	\[
		m_B := \max \{n \in \nn : b_n \in S\} \quad \text{ and } \quad m_C := \max \{n \in \nn : c_n \in S\}.
	\]
	Let $m$ be the maximum of $m_B$ and $m_C$. Since $(b_n + M)_{n \ge 1}$ is an ascending chain of principal ideals of~$M$, if $b \in S \cap B_\epsilon$, then $b_m \mid_M b$. Similarly, if $c \in S \cap C_\epsilon$, then $c_m \mid_M c$. As a consequence, the fact that $a_{m+1}$ is a common divisor of $b_m$ and $c_m$ in $M$ guarantees that $a_{m+1}$ is a common divisor of the set~$S$ in $M$.
	\smallskip
	
	(2) We will prove that the elements $b_1$ and~$c_1$ have no maximal common divisors in $M$. Towards this end, it suffices to fix an arbitrary common divisor $d$ of $b_1$ and $c_1$ in $M$ and argue that the elements $y := b_1-d$ and $z := c_1-d$ have a nonzero common divisor in $M$.
	
	We first show that $d \in \langle A_\epsilon \rangle$. Assume, towards a contradiction, that this is not the case. Since $A_\epsilon \cup B_\epsilon \cup C_\epsilon$ is a generating set of $M$, it follows that $d$ is divisible in $M$ by an element $d_0 \in B_\epsilon \cup C_\epsilon$. Write $d = d_0 + e$ for some $e \in M$. Because $B_\epsilon \cup C_\epsilon$ is a subset of $(1-\epsilon,1)$, it follows that $d_0 > 1-\epsilon$ and $d < 1$, which imply that 
	\[
		e = d - d_0 < 1 - (1-\epsilon) = \epsilon < 1 - \epsilon < \inf B_\epsilon \cup C_\epsilon.
	\]
	Thus, $e$ is not divisible in $M$ by any element of $B_\epsilon \cup C_\epsilon$, and so $e \in \langle A_\epsilon \rangle$. In addition, the fact that $y = (b_1 - d_0) - e < \epsilon < \inf B_\epsilon \cup C_\epsilon$ guarantees that $y \in \langle A_\epsilon \rangle$. One can similarly check that $z \in \langle A_\epsilon \rangle$. Now we can see that $d_0 \notin B_\epsilon$ as, otherwise, the containment $c_1 = d_0 + e + z \in \langle A_\epsilon \cup B_\epsilon \rangle$ would contradict that $v_{\mathsf{d}(c_1)}(x) \ge 0$ for every $x \in \langle A_\epsilon \cup B_\epsilon \rangle$. Similarly, we obtain that $d_0 \notin C_\epsilon$. Hence $d_0 \notin B_\epsilon \cup C_\epsilon$, which is a contradiction. As a result, $d \in \langle A_\epsilon \rangle$.
	
	Finally, we are in a position to show that $y$ and $z$ have a nonzero common divisor in $M$. Observe that~$y$ must be divisible in $M$ by an element $y' \in B_\epsilon \cup C_\epsilon$ as, otherwise, $b_1 = d + y \in \langle A_\epsilon \rangle$. Similarly,~$z$ must be divisible in $M$ by an element $z' \in B_\epsilon \cup C_\epsilon$. Now, in light of part~(1), the set $\{y', z'\}$ must have a common divisor $d'$ that belongs to $A_\epsilon$. Hence $d'$ is a positive common divisor of $y$ and $z$, which concludes our proof.
\end{proof}

The first examples of atomic integral domains that are not strongly atomic were constructed in~\cite[Example~5.2]{mR93}. On the other hand, the first finite-rank torsion-free atomic monoid that is not strongly atomic was constructed in~\cite[Example~4.3]{GV23}. As mentioned in the introduction, the question of whether there exists an atomic Puiseux monoid that is not strongly atomic was posed in the latter \cite[Question~4.4]{GV23} and positively answered in \cite{GLRRT24}. We proceed to show how to use the lifting construction, along with Lemmas~\ref{lem:k-MCD transfer under lift construction} and~\ref{lem:a monoid that is not 2-MCD}, to construct, in a more systematic way, another example of an atomic Puiseux monoid that is not strongly atomic.

\begin{example} \label{ex:an atomic PM that is not strongly atomic}
	Let the monoid $M$ and the sets $A_\epsilon, B_\epsilon$, and $C_\epsilon$ be as in~\eqref{eq:main monoid}. Observe that $A_\epsilon \cup B_\epsilon \cup C_\epsilon$ is a sparing set, and so $M$ is a sparing monoid. Therefore we can take an infinite set $P$ consisting of primes and satisfying that $v_p(q) \ge 0$ for all $p \in P$ and $q \in M^\bullet$. Now, let $S$ be a generating set of $M$ not containing $0$ (for example, we may take $S$ to be $A_\epsilon \cup B_\epsilon \cup C_\epsilon$, or we may just take $S$ to be $M^\bullet$). Since $M$ is not finitely generated, any generating set of $M$ is countable, and so we can take an injection $\pi \colon S \to P$ such that $v_{\pi(s)}(s) = 0$ for all $s \in S$. Now consider the lifting function $\varphi \colon S \to \pp \times \mathcal{N}$ given by the assignments $s \mapsto (\pi(s), \nn_0)$ and the lifting monoid $M_\varphi$ with respect to $\varphi$. For each $s \in S$, it is clear that $\frac s{\pi(s)} \in \mathcal{A}(M_s)$: indeed, $\frac s{\pi(s)} = \min M_s^\bullet$. Therefore the fact that $\frac s{\pi(s)} \in \mathcal{A}(M_s) \setminus \{s\}$, in tandem with part~(2) of Proposition~\ref{prop:atoms in lifting constructions}, guarantees that $\frac s{\pi(s)} \in \mathcal{A}(M_\varphi)$. Furthermore, each $s \in S$ can be written as the sum of $\pi(s)$ copies of $\frac s{\pi(s)}$, and so each element of $S$ is atomic in $M_\varphi$. Since $M = \langle S \rangle$, every element of $M$ is atomic in $M_\varphi$. Moreover, for every $s \in S$, the elements of $M_s^\bullet$ are sums of finitely many copies of $\frac{s}{\pi(s)}$, so the elements of $M_s$ are atomic in~$M_\varphi$. As it follows from Proposition~\ref{prop:unique decomposition} that $M_\varphi$ can be generated by $M \cup \big( \bigcup_{s \in S} M_s \big)$, we conclude that $M_\varphi$ is atomic. Since $M$ is not a $2$-MCD monoid by Lemma~\ref{lem:a monoid that is not 2-MCD}, it follows from Lemma~\ref{lem:k-MCD transfer under lift construction} that $M_\varphi$ is not a $2$-MCD monoid. As a result, $M_\varphi$ is an atomic Puiseux monoid that is not strongly atomic.
\end{example}

We conclude this subsection with the following easy lemma on the atomicity of both $\langle A_\epsilon \rangle$ and $M$, which will be used in the proof of our main theorem (the last part of the lemma justifies why the monoid $M$ was not our candidate in Example~\ref{ex:an atomic PM that is not strongly atomic} for an example of an atomic Puiseux monoid that is not strongly atomic).

\begin{lemma} \label{lem:atoms of the non-atomic monoid M}
		Let the monoid $M$ and the set $A_\epsilon$ be as in~\eqref{eq:main monoid}. Then the following statements hold.
		\begin{enumerate}
			\item The monoid $\langle A_\epsilon \rangle$ is atomic, and $\mathcal{A}(\langle A_\epsilon\rangle) = A_\epsilon$.
			\smallskip
			
			\item For any $i_1, \dots, i_n \in \nn$ with $2 \le i_1 < \dots < i_n$, the element $a_{i_1} + \dots + a_{i_n}$ has a unique additive factorization in $\langle A_\epsilon \rangle$.
			\smallskip
			
			\item The monoid $M$ is not atomic, and $\mathcal{A}(M) = A_\epsilon$.
		\end{enumerate}
\end{lemma}

\begin{proof}
	(1) It suffices to argue that the set of atoms of $\langle A_\epsilon \rangle$ is $A_\epsilon$, which can be readily deduced from the fact that the terms of the sequence $(\mathsf{d}(a_n))_{n \ge 2}$ are pairwise distinct primes.
	\smallskip
	
	(2) Fix $i_1, \dots, i_n \in \nn$ with $2 \le i_1 < \dots < i_n$, and set $b := \sum_{j=1}^n a_{i_j}$. Observe that $b < 1$. In light of part~(1), any additive factorization of $b$ in $\langle A_\epsilon \rangle$ can be written as $\sum_{k=2}^\ell c_k a_k$ for some $c_2, \dots, c_\ell \in \nn_0$ and, after introducing zero coefficients if necessary, we can assume that $\ell \ge \max\{i_1, \dots, i_n\}$. Since $b < 1$, the inequality $c_k < \mathsf{d}(a_k)$ holds for every $k \in \ldb 2,\ell \rdb$. Therefore, for each $k \in \ldb 2,\ell \rdb$, after applying the $\mathsf{d}(a_k)$-adic valuation map on both sides of $ \sum_{j=1}^n a_{i_j} = \sum_{k=2}^\ell c_k a_k$, we see that $c_k = 0$ if $k \notin \{i_1, \dots, i_n\}$. Similarly, for each $j \in \ldb 1,n \rdb$, after applying the $\mathsf{d}(a_{i_j})$-valuation map on both sides of the same equality, we obtain that $c_{i_j} = 1$. Hence $b$ has a unique additive factorization, namely, the formal sum $\sum_{j=1}^n a_{i_j}$.
	\smallskip
	
	(3) Since $v_{\mathsf{d}(b_1)}(x) \ge 0$ for every $x \in \langle A_\epsilon \rangle$, we see that $b_1 \notin \langle A_\epsilon \rangle$. Thus, once we check that $\mathcal{A}(M) = A_\epsilon$, we will also obtain that $M$ is not atomic. Let us argue the desired equality. Because $A_\epsilon \cup B_\epsilon \cup C_\epsilon$ is a generating set of $M$, the equalities $b_n = b_{n+1} + a_{n+1}$ and $c_n = c_{n+1} + a_{n+1}$ for every $n \in \nn$ guarantee that $\mathcal{A}(M) \subseteq A_\epsilon$. On the other hand, the fact that no element of $B_\epsilon \cup C_\epsilon$ can divide any element of $A_\epsilon$ in $M$ (because $\sup A_\epsilon < \frac12 < \inf B_\epsilon \cup C_\epsilon$), together with part~(1), guarantees that $A_\epsilon \subseteq \mathcal{A}(M)$.
\end{proof}

\medskip
\subsection{On The Ascent of Atomicity to Monoid Algebras}

This final section is dedicated to establish our main theorem: we construct a rank-one torsion-free atomic monoid whose monoid algebras (over \emph{any} commutative ring) are not atomic. In Example~\ref{ex:an atomic PM that is not strongly atomic}, we have already exhibited a rank-one torsion-free atomic monoid that is not strongly atomic. However, the monoid we need in our main theorem should be constructed by choosing a more subtle lifting function for the monoid $M$ in~\eqref{eq:main monoid}. For the convenience of the readers, we restate our main theorem (also stated in the introduction).

\begin{main-theorem}
	There exists a rank-one torsion-free atomic monoid $M$ such that the monoid algebra $R[M]$ is not atomic for any integral domain~$R$.
\end{main-theorem}

\begin{proof}
	Let the monoid $M$ and the sequences $(a_n)_{n \ge 2}$, $(b_n)_{n \ge 1}$, and $(c_n)_{n \ge 1}$ be as defined in the paragraph before Lemma~\ref{lem:a monoid that is not 2-MCD}: the corresponding underlying sets of $(a_n)_{n \ge 2}$, $(b_n)_{n \ge 1}$, and $(c_n)_{n \ge 1}$ are $A_\epsilon$, $B_\epsilon$, and~$C_\epsilon$, and the monoid $M$ is generated by the set $A_\epsilon \cup B_\epsilon \cup C_\epsilon$. Also, recall that $\epsilon < \frac1{10}$ while
	\[
		\sum_{n \ge 2} a_n < \frac{\epsilon}8 \quad \text{ and } \quad 1 - \frac{\epsilon}4 < \inf (B_\epsilon \cup C_\epsilon) < \max\{b_1, c_1\}  < 1.
	\]
	Now take $\delta \in \rr_{> 0}$ such that
	\begin{equation} \label{eq:delta}
		\delta < \min \Big\{  \inf (B_\epsilon \cup C_\epsilon) - \Big(1 - \frac{\epsilon}4 \Big), 1 - \max\{b_1, c_1\} \Big\}.
	\end{equation}
	
	The first step of this proof is to find a lifting function $\varphi \colon S \to \pp \times \mathcal{N}$ that yields a convenient lifting monoid of $M$. To do so, set $S := B_\epsilon \cup C_\epsilon$. Since $M$ is sparing, we can take an infinite set~$P$ consisting of odd primes and satisfying the inequality $v_p(q) \ge 0$ for all $p \in P$ and $q \in M^\bullet$. Let $(s_n)_{n \ge 1}$ be a sequence of pairwise distinct elements whose underlying set is $S$. For every $n \in \nn$, we can take $p_n \in P$ large enough so that $v_{p_n}(s_n) = 0$ and $\frac{s_n}{p_n} <  \delta$, and then take $h_n \in \ldb 2, p_n - 2 \rdb$ such that $h_n \frac{s_n}{p_n}$ belongs to a $\delta$-neighborhood of $\frac12 - \epsilon$ if $s_n \in B_\epsilon$ and to a $\delta$-neighborhood of $\frac12$ if $s_n \in C_\epsilon$. We can further assume that the sequence $(p_n)_{n \ge 1}$ is strictly increasing, and set $k_n := p_n - h_n \in \ldb 2, p_n - 2 \rdb$ for every $n \in \nn$. When $s_n \in B_\epsilon$ one can use the fact that $h_n \frac{s_n}{p_n}$ belongs to a $\delta$-neighborhood of $\frac12 - \epsilon$ to infer that
	\begin{equation} \label{eq:aux eq for B}
		\Big{|} h_n \frac{s_n}{p_n} - \Big(\frac 12 - \epsilon \Big) \Big{|} < \frac{\epsilon}2 \quad \text{ and } \quad \Big{|} k_n \frac{s_n}{p_n} - \Big(\frac 12 + \epsilon \Big) \Big{|} < \frac{\epsilon}2
	\end{equation}
	(the first inequality is clear because $\delta < \frac{\epsilon}4$, while the second one can be obtained by using the inequalities $1 - \frac{\epsilon}4 + \delta < s_n < 1 - \delta$ and the equality $h_n + k_n = p_n$). When $s_n \in C_\epsilon$, we can similarly arrive to the inequalities
	\begin{equation} \label{eq:aux eq for C}
		\Big{|} h_n \frac{s_n}{p_n} - \frac 12 \Big{|} < \frac{\epsilon}4 \quad \text{ and } \quad \Big{|} k_n \frac{s_n}{p_n} - \frac 12 \Big{|} < \frac{\epsilon}4.
	\end{equation}
	From this point on, we set $p_s := p_n$, $h_s := h_n$, and $k_s := k_n$ for all $s \in S$ and $n \in \nn$ such that $s = s_n$ (this will ease the notation throughout the rest of the proof).
	
	Now consider the function $\pi \colon S \to \pp$ defined by $\pi(s) := p_s$ for every $s \in S$. Observe that $\pi$ is an injection and also that $v_{\pi(s)}(s) = 0$ for all $s \in S$. For each $s \in S$, consider the numerical monoid $N_s := \langle h_s, k_s \rangle$, which contains $\pi(s) $ because $h_s + k_s = \pi(s)$. Now define the lifting function $\varphi \colon S \to \pp \times \mathcal{N}$ by setting $\varphi(s) := (\pi(s), N_s)$ for every $s \in S$, and let $M_\varphi$ be the lifting monoid of $M$ with respect to~$\varphi$. For each $s \in S$, set $H_s := h_s \frac{s}{\pi(s)}$ and $K_s := k_s \frac{s}{\pi(s)}$, which implies that $M_s = \langle H_s, K_s \rangle$. Note that $H_s + K_s = s$ for all $s \in S$. We can now rewrite the inequalities in~\eqref{eq:aux eq for B} and~\eqref{eq:aux eq for C} as follows:
	\begin{equation} \label{eq:aux interval for s in B}
		\Big{|} H_s - \Big(\frac 12 - \epsilon \Big) \Big{|} < \frac{\epsilon}2 \quad \text{ and } \quad \Big{|} K_s - \Big(\frac 12 + \epsilon \Big) \Big{|} < \frac{\epsilon}2
	\end{equation}
	if $s \in B_\epsilon$, while
	\begin{equation} \label{eq:aux interval for s in C}
		\Big{|} H_s - \frac 12 \Big{|} < \frac{\epsilon}4 \quad \text{ and } \quad \Big{|} K_s - \frac 12 \Big{|} < \frac{\epsilon}4
	\end{equation}
	if $s \in C_\epsilon$. Since $\epsilon < \frac1{10}$, it follows from the inequalities in~\eqref{eq:aux interval for s in B} and~\eqref{eq:aux interval for s in C} that $\min M_s^\bullet > \frac12 - \frac{3}{2} \epsilon > \frac7{20} > \frac 13$ for all $s \in S$. Our next step is to argue that the lifting monoid $M_\varphi$ is atomic. To do so, we first need to verify the following claim.
	\smallskip
	
	\noindent \textsc{Claim 1.} $M_\varphi \cap \qq_{\le 1/3} \subseteq \langle A_\epsilon \rangle$.
	\smallskip
	
	\noindent \textsc{Proof of Claim 1.} Take $q \in M_\varphi$ such that $q \le \frac13$. For each $s \in S$, the inequality $\min M_s^\bullet > \frac 13$ guarantees that no nonzero element in $M_s$ divides $q$ in $M_\varphi$. Since $M_\varphi = \big\langle \bigcup_{r \in M} M_r \big\rangle$, we see that $q \in M$. In addition, the chain of inequalities $\inf B_\epsilon \cup C_\epsilon > 1 - \epsilon > \frac 13$ implies that $q$ is not divisible in $M$ by any element of $B_\epsilon \cup C_\epsilon$. This, along with the fact that $A_\epsilon \cup B_\epsilon \cup C_\epsilon$ is a generating set of~$M$, ensures that $q \in \langle A_\epsilon \rangle$. Hence $M_\varphi \cap \qq_{\le 1/3} \subseteq \langle A_\epsilon \rangle$, and the claim is established.
	\smallskip
	
	We proceed to prove that $M_\varphi$ is atomic. For each $s \in S$, the fact that $h_s + k_s$ is prime implies that $\gcd(h_s, k_s) = 1$, and so $\mathcal{A}(M_s) = \{H_s, K_s\}$. Because $s \notin \{H_s, K_s\}$, part~(2) of Proposition~\ref{prop:atoms in lifting constructions} guarantees that $H_s, K_s \in \mathcal{A}(M_\varphi)$. Thus, for each $s \in S$, every element of $M_s$ is atomic in $M_\varphi$. In particular, every $s \in B_\epsilon \cup C_\epsilon$ is atomic in $M_\varphi$. On the other hand, the inequality $\sup A_\epsilon \le \frac13$, in tandem with Claim~1, implies that every element of $M_\varphi$ dividing an element of $A_\epsilon$ in $M_\varphi$ must belong to $\langle A_\epsilon \rangle$, and so the inclusion $A_\epsilon \subseteq \mathcal{A}(M_\varphi)$ follows from the equality $\mathcal{A}(\langle A_\epsilon \rangle) = A_\epsilon$ in part~(1) of Lemma~\ref{lem:atoms of the non-atomic monoid M}. Hence from the fact that $A_\epsilon \cup B_\epsilon \cup C_\epsilon$ is a generating set of $M$, we obtain that every element of $M$ is atomic in $M_\varphi$. As $M_\varphi$ is generated by the set $M \cup  \big( \bigcup_{s \in S} M_s \big)$, we further obtain that $M_\varphi$ is an atomic monoid with
	\begin{equation} \label{eq:atoms of the lifting monoid}
		\mathcal{A}(M_\varphi) := A_\epsilon \cup \big\{ H_s, K_s : s \in B_\epsilon \cup C_\epsilon  \big\}.
	\end{equation}
	
	The rest of the proof is dedicated to argue that $R[M_{\varphi}]$ is not an atomic integral domain for any commutative ring $R$. To do so, fix an arbitrary commutative ring $R$. If $R$ is not an integral domain, then $R[M_\varphi]$ is not an integral domain \cite[Theorem~8.1]{rG84}, and we are done. Thus, we will assume that~$R$ is an integral domain. 
	
	Proving that the monoid algebra $R[M_\varphi]$ is not atomic amounts to arguing that the polynomial expression $f(X) := X^{b_1} + X^{c_1} \in R[M_\varphi]$ cannot be factored into irreducibles. Assume, without loss of generality, that $b_1 > c_1$. Now suppose, by way of contradiction, that
	\begin{equation} \label{eq:factorization}
		f(X) = a_1(X) \cdots a_\ell(X),
	\end{equation}
	where $\ell \ge 2$ and $a_1(X), \dots, a_\ell(X)$ are irreducibles in $R[M_\varphi]$. Observe that, for each $i \in \ldb 1, \ell \rdb$, the polynomial expression $a_i(X)$ does not belong to $R$ because otherwise $a_i(X)$ would divide the leading coefficient of $f(X)$, namely, $a_i(X) \mid_{R[M_\varphi]} 1$, and so $a_i(X)$ would be a unit of $R[M_\varphi]$. Hence $\deg a_i(X) > 0$ for every $i \in \ldb 1, \ell \rdb$. Set $d_i := \deg a_i(X)$ and $o_i := \text{ord} \, a_i(X)$ for every $i \in \ldb 1, \ell \rdb$. Then $b_1 = \sum_{i=1}^\ell d_i$ and $c_1 = \sum_{i=1}^\ell o_i$. We split the rest of the proof into two cases.
	\smallskip
	
	\noindent \textsc{Case 1:} $o_i \in M$ for every $i \in \ldb 1,\ell \rdb$. Since $v_{\mathsf{d}(c_1)}(q) \ge 0$ for any $q \in \langle A_\epsilon \cup B_\epsilon \rangle$, after applying $\mathsf{d}(c_i)$-adic valuation on both sides of $c_1 = \sum_{i=1}^\ell o_i$, we find that $o_j \notin \langle A_\epsilon \cup B_\epsilon \rangle$ for some $j \in \ldb 1, \ell \rdb$. Assume, without loss of generality, that $j = 1$. Since $M$ is generated by $A_\epsilon \cup B_\epsilon \cup C_\epsilon$, we see that $o_1$ must be divisible in~$M$ by an element $c_k \in C_\epsilon$. Thus, $c_k \le o_1 \le c_1$, which implies that $o_1 \in (1- \epsilon, 1)$. This, along with the inequalities $o_1 \le d_1 \le b_1 < 1$, guarantees that $d_1 \in (1-\epsilon, 1)$. Therefore every element in the support of $a_1(X)$ must belong to $(1-\epsilon, 1)$, and so the fact that $a_1(X)$ is irreducible implies that $a_1(X)$ is not a monomial (as the degree of each irreducible monomial of $R[M_\varphi]$ belongs to $\mathcal{A}(M_\varphi)$ and so is less than $1-\epsilon$). As $\sum_{i=1}^\ell d_i = b_1 < 1$, the fact that $d_1 > 1-\epsilon$ ensures that $d_i < \epsilon < \frac1{10}$ for every $i \in \ldb 2, \ell \rdb$, and so every element in $\bigcup_{i=2}^\ell \, \text{supp} \, a_i(X)$ belongs to $[0, \frac1{10})$, and so it follows from Claim~1 that $\bigcup_{i=2}^\ell \, \text{supp} \, a_i(X) \subseteq \langle A_\epsilon \rangle$. 
	
	Since $a_1(X)$ is irreducible in $R[M_\varphi]$, there must be an element of $\text{supp} \, a_1(X)$ that is not divisible in $M_\varphi$ by any element of $B_\epsilon \cup C_\epsilon$, as otherwise, part~(1) of Lemma~\ref{lem:a monoid that is not 2-MCD} would guarantee the existence of $a \in A_\epsilon$ such that $X^a$ properly divides $a_1(X)$ in $R[M_\varphi]$ (because $a_1(X)$ is not a monomial). We consider the following (not necessarily mutually exclusive) subcases.
	\smallskip
	
	\noindent \textsc{Case 1.1:} $\text{supp} \, a_1(X) \subseteq M$. As mentioned in the previous paragraph, there exists $d \in \text{supp} \, a_1(X)$ that is not divisible in $M_\varphi$ by any element of $B_\epsilon \cup C_\epsilon$. As a consequence, the fact that $A_\epsilon \cup B_\epsilon \cup C_\epsilon$ is a generating set of $M$ implies that $d \in \langle A_\epsilon \rangle$. We can assume, without loss of generality, that $d$ is the maximum of $\langle A_\epsilon \rangle \cap \text{supp} \, a_1(X)$. Because $d_2, \dots, d_\ell \in \langle A_\epsilon \rangle$, it follows that the element $d' := d + \sum_{i=2}^\ell d_i$ belongs to $\langle A_\epsilon \rangle$, which implies that $d' \notin \{b_1, c_1\}$ and so that $b_1 > d' > c_1$. Therefore, after unfolding the right-hand side of \eqref{eq:factorization}, at least two terms of degree $d'$ must be produced so that they can cancel out (here we are using the fact that $R$ is an integral domain). Thus, there must be another term of degree $d'$ (besides that suggested by the defining equality $d' := d + \sum_{i=2}^\ell d_i$), which has to be produced using, as a factor, a monomial term of $a_1(X)$ whose degree $d'_1$ is strictly greater than~$d$. It follows from the maximality of $d$ that $d'_1 \notin \langle A_\epsilon \rangle$. Since $M_\varphi$ is atomic, it follows from~\eqref{eq:atoms of the lifting monoid} that either $H_s$ or $K_s$ must divide $d'_1$ in $M_\varphi$ for some $s \in S$. Since $d'_1 \mid_{M_\varphi} d'$, it is also true that either $H_s$ or $K_s$ divides $d'$ in $M_\varphi$. Assume first that $H_s \mid_{M_\varphi} d'$. Since $H_s$ is an $M_s$-projection and $d' \in M$, it follows from part~(2) of Proposition~\ref{prop:divisibility relation of M-projections} that $s \mid_M d'$. Now, after writing $d' = s + t$ for some $t \in M$, we see that $t$ must be divisible in $M$ by an element of $B_\epsilon \cup C_\epsilon$ as, otherwise, $t \in \langle A_\epsilon \rangle$ and so $v_{\mathsf{d}(s)}(s) = v_{\mathsf{d}(s)}(d' - t) \ge 0$, which is not possible: as a result, $d' = s+t > 2(1-\epsilon) > 1$, which contradicts that $d' \in (c_1, b_1)$. We can similarly generate a contradiction if we assume that $K_s \mid_{M_\varphi} d'$.
	\smallskip
	
	\noindent \textsc{Case 1.2:} $\text{supp} \, a_1(X) \not\subseteq M$. Take $d$ to be the largest exponent in $\text{supp} \, a_1(X)$ such that $d \notin M$. Observe that, for some $s \in S$, the $M_s$-projection in the $\varphi$-lifting decomposition of $d$ is nonzero. Because $d_2, \dots, d_\ell \in \langle A_\epsilon \rangle \subseteq M$, the element $d' := d + \sum_{i=2}^\ell d_i$ also has a nonzero $M_s$-projection in its $\varphi$-lifting decomposition, and so $d' \notin M$. Therefore the fact that $b_1, c_1 \in M$ implies that $b_1 > d' > c_1$. As in the previous case, after unfolding the right-hand side of~\eqref{eq:factorization}, at least two terms of degree~$d'$ must be produced, and so there exists $d'_1 \in \text{supp} \, a_1(X)$ with $d'_1 > d$ such that $d'_1 \mid_{M_\varphi} d'$. The maximality of $d$ now implies that $d'_1 \in M$. In this case, we would be able to find $d'_i \in \text{supp} \, a_i(X)$ for every $i \in \ldb 2,\ell \rdb$ such that $d' = \sum_{i=1}^\ell d'_i$. However, from $d'_1 \in M$ and $d'_2, \dots, d'_\ell \in \bigcup_{i=2}^\ell \, \text{supp} \, a_i(X) \subseteq \langle A_\epsilon \rangle \subseteq M$, we obtain a contradiction with the fact that $d' \notin M$.
	\medskip	
	
	\noindent \textsc{Case 2:} $o_j \notin M$ for some $j \in \ldb 1,\ell \rdb$. In order to generate the needed contradiction for this case, we first need to establish the following claim.
	\smallskip
	
	\noindent \textsc{Claim 2.} $2 a \nmid_{M_\varphi} b_1$ and $2 a \nmid_{M_\varphi} c_1$ for any $a \in A_\epsilon$.
	\smallskip
	
	\noindent \textsc{Proof of Claim 2.} Suppose, towards a contradiction, that $2a_k \mid_{M_\varphi} b_1$ for some $k \in \nn_{\ge 2}$. Since both $2 a_k$ and $b_1$ belong to $M$, it follows from part~(1) of Proposition~\ref{prop:divisibility relation of M-projections} that $2a_k \mid_M b_1$. Then we can write $b_1 = 2a_k + e$ for some $e \in M$, and after applying $v_{\mathsf{d}(b_1)}$-adic valuation map on both sides of the latter equality, we find that $e$ is divisible by an element of $B_\epsilon$, namely, $b_j$ for some $j \ge 2$. In light of~\eqref{eq:def of the b_j's and c_j_s}, we can write $b_1 - b_j = a_2 + \dots + a_j$. It follows from part~(2) of Lemma~\ref{lem:atoms of the non-atomic monoid M} that the element $a_2 + \dots + a_j$ has a unique additive factorization in $\langle A_\epsilon \rangle$. Thus,
	\[
		2a_k + (e - b_j) = b_1 - b_j = a_2 + \dots + a_j < \frac{\epsilon}8 < \inf B_\epsilon \cup C_\epsilon,
	\]
	and so no element of $B_\epsilon \cup C_\epsilon$ can divide $e - b_j$ in $M$, which implies that $e - b_j \in \langle A_\epsilon \rangle$. Therefore, after writing $e - b_j$ as a sum of atoms in the atomic monoid $\langle A_\epsilon \rangle$, we obtain an additive factorization of $a_2 + \dots + a_j$ in $\langle A_\epsilon \rangle$ having at least two copies of the atom $a_k$, which contradicts that $a_2 + \dots + a_j$ has a unique additive factorization in $\langle A_\epsilon \rangle$. The fact that  $2 a \nmid_{M_\varphi} c_1$ for any $a \in A_\epsilon$ can be handled similarly. Thus, the claim is established.
	\smallskip
	
	Now fix an index $j \in \ldb 1, \ell \rdb$ such that $o_j \notin M$. Then there exists $s \in S$ such that the $M_s$-projection $\mathsf{p}_s(o_j)$ in the $\varphi$-lifting decomposition of~$o_j$ is nonzero. Since $\sum_{i=1}^\ell o_i = c_1 \in M$, we can take $k \in \ldb 1,\ell \rdb \setminus \{j\}$ such that the $M_s$-projection $\mathsf{p}_s(o_k)$ in the $\varphi$-lifting decomposition of $o_k$ is also nonzero. Therefore it follows from the inequalities in~\eqref{eq:aux interval for s in B} and~\eqref{eq:aux interval for s in C} that
	\[
		o_j \ge \mathsf{p}_s(o_j) \ge \min M_s^\bullet > \frac12 - \frac32 \epsilon \quad \text{ and } \quad 	o_k \ge \mathsf{p}_s(o_k) \ge \min M_s^\bullet > \frac12 - \frac32 \epsilon.
	\]
	This, along with $o_j + o_k \le c_1 < 1$, ensures that both $o_j$ and $o_k$ belong to the interval $\big( \frac12 - \frac32 \epsilon, \frac12 + \frac32 \epsilon \big)$. As a result, the fact that $\mathsf{p}_s(o_j), \mathsf{p}_s(o_k) \in M_s^\bullet$ implies that $\mathsf{p}_s(o_j), \mathsf{p}_s(o_k) \in \{H_s, K_s\}$, which is in turn a subset of $\big( \frac12 - \frac32 \epsilon, \frac12 + \frac32 \epsilon\big)$ because $H_s + K_s = s < 1$. Hence, for any $r \in S \setminus \{s\}$, the $M_r$-projection $\mathsf{p}_r(o_j)$ in the $\varphi$-lifting decomposition of $o_j$ satisfies $\mathsf{p}_r(o_j) \le o_j - \mathsf{p}_s(o_j) < 3\epsilon < \frac3{10} < \min M_r^\bullet$ (as we have already seen that $\min M_r^\bullet > \frac13$), and so $\mathsf{p}_r(o_j) = 0$. In a similar way, we can check that, for any $r \in S \setminus \{s\}$, the $M_r$-projection in the $\varphi$-lifting decomposition of $o_k$ is zero. Now observe that since $c_1$ is not divisible in $M$ by any element of $B_\epsilon$, it follows from part~(2) of Proposition~\ref{prop:divisibility relation of M-projections} that $s \in C_\epsilon$.
	
	As we just did with the orders $o_j$ and $o_k$, let us now obtain some information about the degrees $d_j$ and~$d_k$ as well as their corresponding projections. Since $1 > b_1 \ge d_j + d_k$, the inequalities $d_j > o_j > \frac12 - \frac32 \epsilon$  and $d_k > o_k > \frac12 - \frac32 \epsilon$ guarantee that $d_j, d_k \in \big( \frac12 - \frac32 \epsilon, \frac12 + \frac32 \epsilon \big)$. Hence, for each  $i \in \ldb 1, \ell \rdb \setminus \{j,k\}$, the fact that $\sum_{i=1}^\ell d_i = b_1 < 1$ implies that
	\[
		\max \, \text{supp} \, a_i(X) = d_i \le 1 - (d_j + d_k) < 3\epsilon < \frac3{10} < \frac13,
	\]
	and so $\text{supp} \, a_i(X) \subseteq \langle A_\epsilon \rangle_{< 1/3}$ by virtue of Claim~1. In addition, the fact that $d_j, d_k \in \big( \frac12 - \frac32 \epsilon, \frac12 + \frac32 \epsilon \big)$ implies that neither $d_j$ nor $d_k$ are divisible in $M_\varphi$ by any element of $B_\epsilon \cup C_\epsilon$. Also, it follows from Claim~2 that $2a \nmid_{M_\varphi} d_j$ and $2a \nmid_{M\varphi} d_k$ for any $a \in A_\epsilon$. This, together with the inequality $\sum_{n \ge 2} a_n < \frac{\epsilon}8$, guarantees that
	\begin{equation} \label{eq:two equalities for p(d_j) and p(d_k)}
		\mathsf{p}(d_j) < \frac{\epsilon}4 \quad \text{and} \quad \mathsf{p}(d_k) < \frac{\epsilon}4,
	\end{equation}
	where $\mathsf{p}(d_j)$ and $\mathsf{p}(d_k)$ are the $M$-projections in the $\varphi$-lifting decompositions of $d_j$ and $d_k$, respectively. Therefore, as $d_j \in \big( \frac12 - \frac32 \epsilon, \frac12 + \frac32 \epsilon \big)$, there exists a unique $t \in S$ such that the $M_t$-projection $\mathsf{p}_t(d_j)$ in the $\varphi$-lifting decomposition of $d_j$ is nonzero. We can similarly conclude that there exists a unique $t' \in S$ such that the $M_{t'}$-projection $\mathsf{p}_{t'}(d_k)$ in the $\varphi$-lifting decomposition of $d_k$ is nonzero. Now, because $\sum_{i=1}^\ell d_i = b_1 \in M$, the fact that $d_i \in \langle A_\epsilon \rangle \subset M$ for every $i \in \ldb 1, \ell \rdb \setminus \{j,k\}$ enforces the equality $t' = t$. Thus, we can write
	\begin{equation} \label{eq:d_j and d_k}
		d_j = \mathsf{p}(d_j) + \mathsf{p}_t(d_j) \quad \text{ and } \quad d_k = \mathsf{p}(d_k) + \mathsf{p}_t(d_k),
	\end{equation}
	where $\mathsf{p}(d_j), \mathsf{p}(d_k) \in \langle A_\epsilon \rangle_{< \epsilon/8}$. Since~$b_1$ is not divisible in $M$ by any element of $C_\epsilon$, it follows from part~(2) of Proposition~\ref{prop:divisibility relation of M-projections} that $t \in B_\epsilon$.
	
	We are in a position to finish our proof. For each $i \in \ldb 1, \ell \rdb \setminus \{j,k\}$, the inclusion $\text{supp} \, a_i(X) \subset \langle A_\epsilon \rangle$ ensures that the $M_t$-projection of $d_i$ in its $\varphi$-lifting decomposition equals zero. Now, as the $M_t$-projection of $b_1 = \sum_{i=1}^\ell d_i$ equals zero, we see that $\mathsf{p}_t(d_j) + \mathsf{p}_t(d_k) \in \nn t$, and so the fact that $\mathsf{p}_t(d_j) + \mathsf{p}_t(d_k) \le b_1 < 1$ implies that $\mathsf{p}_t(d_j) + \mathsf{p}_t(d_k) = t$. Hence the equality
	\[
		\{\mathsf{p}_t(d_j), \mathsf{p}_t(d_k)\} = \{H_t, K_t\}
	\]
	holds. We have previously observed that $\mathsf{p}_s(o_j), \mathsf{p}_s(o_k) \in \{H_s, K_s\}$ and also that $s \in C_\epsilon$ and $t \in B_\epsilon$. As $t \in B_\epsilon$, it follows from~\eqref{eq:aux interval for s in B} that $H_t$ belongs to an $\frac{\epsilon}2$-neighborhood of $\frac12 - \epsilon$ while $K_t$ belongs to an $\frac{\epsilon}2$-neighborhood of $\frac 12 + \epsilon$. On the other hand, as $s \in C_\epsilon$, it follows from~\eqref{eq:aux interval for s in C} that both $H_s$ and $K_s$ belong to an $\frac{\epsilon}4$-neighborhood of $\frac 12$. If $\mathsf{p}_t(d_j) = H_t$, then it follows from~\eqref{eq:two equalities for p(d_j) and p(d_k)} and~\eqref{eq:d_j and d_k} that 
	\[
		d_j = \mathsf{p}(d_j) + \mathsf{p}_t(d_j) <  \frac{\epsilon}4 + H_t < \frac12 - \frac{\epsilon}4 < \mathsf{p}_s(o_j) < o_j,
	\]
	which is a contradiction as $d_j$ and $o_j$ are the degree and order, respectively, of the same polynomial expression of $R[M_\varphi]$. We can arrive to a similar contradiction if we assume that $\mathsf{p}_t(d_k) = H_t$.
\end{proof}

\begin{remark}
	The monoid $M_\varphi$ in the proof of our main theorem is not strongly atomic because, by virtue of Lemmas~\ref{lem:k-MCD transfer under lift construction} and~\ref{lem:a monoid that is not 2-MCD}, it is not a 2-MCD monoid. Hence $M_\varphi$ is another example of an atomic Puiseux monoid that is not strongly atomic (see Example~\ref{ex:an atomic PM that is not strongly atomic}).
\end{remark}

The only known examples of integral domains that are Furstenberg but not atomic are the ring $\text{Hol}(\cc)$ consisting of all entire functions (first pointed out in \cite[Example~4.19]{pC17}) and rings of integer-valued polynomials (see \cite[Lemma~16]{nL19} and \cite[Example~5.2]{GL22}). It is well known that the ring $\text{Hol}(\cc)$ has infinite dimension and also that rings of integer-valued polynomials have dimension at least~$2$. We conclude this paper with an example of a one-dimensional monoid algebra that is Furstenberg but not atomic.

\begin{example}
	Let $M_\varphi$ be the monoid in the proof of our main theorem and set $M_A := \langle A_\epsilon \rangle$. One can readily argue that the monoid $M_d := \big\langle \frac{1}{\mathsf{d}(a_n)} : n \ge 2 \big\rangle$ satisfies the ACCP. This, along with the fact that $M_d$ is reduced, ensures that every submonoid of $M_d$ satisfies the ACCP. Hence $M_A$ satisfies the ACCP. It follows from~\cite[Theorem~13]{AJ15} that the monoid algebra of $M_A$ over any field also satisfies the ACCP and is, therefore, a Furstenberg domain.
	
	It turns out that if we consider the monoid algebra of $M_\varphi$ (instead of $M_A$) over any field, we obtain a Furstenberg domain that is not atomic. To argue this, fix a field $F$. Since $M_\varphi$ is a Puiseux monoid, it has rank $1$ and so the monoid algebra $F[M_\varphi]$ is one-dimensional. We have already seen in the proof of our main theorem that $F[M_\varphi]$ is not atomic. To argue that $F[M_\varphi]$ is a Furstenberg domain, assume towards a contradiction, that there exists a nonzero nonunit polynomial expression $g(X)$ in $F[M_\varphi]$ that is not divisible by any irreducible. Therefore there exist $\ell \in \nn$ with $\ell > 3\deg g(X)$ and nonunits $g_1(X), \dots, g_\ell(X) \in F[M_\varphi]$ such that $g(X) = g_1(X) \cdots g_\ell(X)$. Because $\ell > 3\deg g(X)$, we can take $i \in \ldb 1,\ell \rdb$ such that $\deg g_i(X) < \frac13$. Since $g(X)$ is not divisible by any irreducible in $F[M_\varphi]$, the same holds for $g_i(X)$. Thus, after replacing $g(X)$ by $g_i(X)$, we can assume that $\deg g(X) < \frac13$.  Therefore $\text{supp} \, g(X) \subseteq M_\varphi \cap \qq_{< 1/3}$, and so it follows from Claim~1 in the proof of our main theorem that $g(X) \in F[M_A]$. Since $F[M_A]$ is a Furstenberg domain, there exists an irreducible $a(X)$ of $F[M_A]$ such that $a(X) \mid_{F[M_A]} g(X)$. Now observe that by Claim~1 in the proof of our main theorem every polynomial expression in $F[M_\varphi]$ with degree less than $\frac13$ must belong to $F[M_A]$, which implies that each irreducible of $F[M_A]$ with degree less than $\frac13$ is also an irreducible of $F[M_\varphi]$. Hence $a(X)$ is an irreducible of $F[M_\varphi]$ that divides $g(X)$, which is a contradiction. Hence we conclude that $F[M_\varphi]$ is a Furstenberg domain.
\end{example}

\bigskip
\section*{Acknowledgments}

During the preparation of the present paper both authors were part of the PRIMES program at MIT, and they would like to thank the directors and organizers of the program for making this collaboration possible. Both authors are grateful to an anonymous referee for helpful comments. The first author kindly acknowledges support from NSF under the award DMS-2213323.

\bigskip
\section*{Conflict of Interest Statement}

On behalf of all authors, the corresponding author states that there is no conflict of interest related to this paper.

\bigskip

\end{document}